\documentclass{amsart}

\usepackage{amssymb}
\usepackage{url}

\usepackage{graphicx}

\newtheorem{theorem}{Theorem}[section]
\newtheorem{lemma}[theorem]{Lemma}
\newtheorem{corollary}[theorem]{Corollary}
\newtheorem{conjecture}[theorem]{Conjecture}

\newtheorem{notation}[theorem]{Notation}
\newtheorem{example}[theorem]{Example}
\newtheorem{definition}[theorem]{Definition}

\def\st {\ast}
\def\difst {\divideontimes}
\def\capst {\widehat{\cap}}
\def\lch {\triangleleft}
\def\rch {\triangleright}
\def\inch {\in}
\def\difch {\neq}

\def\totdif {\not\equiv}

\begin{document}

%%%%%%%%%%%%%%%%%%%%%%%%%
% Subject classification
%%%%%%%%%%%%%%%%%%%%%%%%%

% Provide an AMS subject classification with one or two primary classification
% numbers and, optionally, one or more secondary classification numbers.
% Use the following format:  "Primary 42B25. Secondary 42B60, 20E26"


%%%%%%%%%
% Title
%%%%%%%%%

% Title, in lower case, with no explicit linebreaks (\\).  If the title
% is too long to be used as a running head, add a short version of the
% title in brackets, as in \title[shorttitle]{fulltitle}.

\title[$\sigma$-Set Theory]{$\sigma$-Set Theory: Introduction to the concepts of $\sigma$-antielement, $\sigma$-antiset and Integer Space.}

%%%%%%%%%%%%%%%%%%%%%%%%%%%%%%
% Author names and addresses
%%%%%%%%%%%%%%%%%%%%%%%%%%%%%%

% Provide one separate \author{...} \address{...} \email{....} entry for each
% author, i.e., do not combine multiple authors.  Separate address lines by double
% slashes.  Do not attach footnotes to author  names. (For acknowledgements use
% the "\thanks" construct below.)
%

\author{Ivan Gatica Araus}

\address{ Becario MAE-AECI}

\address{ Department of Mathematical Analysis, University of Sevilla, St.
Tarfia s/n, Sevilla, SPAIN}

\address{ Department of Mathematics, University Andr$\acute{e}$s Bello, Los Fresnos 52, Vi$\tilde{n}$a del mar, CHILE}

\email{igatica@us.es}

%%%%%%%%%%%%%%%%%%%%
% Acknowledgements
%%%%%%%%%%%%%%%%%%%

% Use \thanks for acknowledgements as footnotes to the title page.
% (Note that footnotes inside \author or \title macros are not
% allowed.)
%
% In case of multiple author papers, phrase the acknowledgement to
% say "The first author was supported by ...  The second author was
% supported by ..."

\thanks{During the development of this work, we had the initial support of Victor Cardenas Vera, Diego Lobos Maturana, Mariela Carvacho Bustamente and Hector Carrasco Altamirano, which motivated us to move forward. We also received help from Rafael Espinola Garcia, Carlos Jim$\acute{e}$nez G$\acute{o}$mez, Carlos Hernandez Linares, Andreea Sarafoleanu and Francisco Canto Martin in the correction and translation of the manuscript. Without this help we would not have been able to finish our work.}

%%%%%%%%%%%%%
% Abstract
%%%%%%%%%%%%%
%
% Abstracts should not contain macros (so that they can be processed independently
% of the paper.) Avoid displayed math and references in the abstract.

\begin{abstract}

In this paper we develop a theory called $\sigma$-Set Theory, in which we present an axiom system developed from the study of Set Theories of Zermelo-Fraenkel, Neumann-Bernays-Godel and Morse-Kelley. In $\sigma$-Set Theory, we present the proper existence of objects called $\sigma$-antielement, $\sigma$-antiset, natural numbers, antinatural numbers and generated $\sigma$-set by two $\sigma$-sets, from which we obtain, among other things, a commutative non-associative algebraic structure called Integer Space $3^{X}$, which corresponds to the algebraic completion of $2^{X}$.

\end{abstract}

\maketitle

%%%%%%%%%%%%%%%%%%%%%%%%%%%%%%%%%%%%%%%%%%%%%%%%%%%%%%%%%%%%%%%%%%%%%%%%%
% end Topmatter
%%%%%%%%%%%%%%%%%%%%%%%%%%%%%%%%%%%%%%%%%%%%%%%%%%%%%%%%%%%%%%%%%%%%%%%%%

%%%%%%%%%%%%%%%%%%%%%%%%%%%%%%%%%%%%%%%%%%%%%%%%%%%%%%%%%%%%%%%%%%%%%%%%%
% body of paper
%%%%%%%%%%%%%%%%%%%%%%%%%%%%%%%%%%%%%%%%%%%%%%%%%%%%%%%%%%%%%%%%%%%%%%%%%

\begin{section}{Introduction}

The $\sigma$-Set Theory comes from the study of multivalued operators $T:X\rightarrow 2^{X}$ defined in a Banach Space $X$ with images in
$2^{X}$. From here it seemed interesting to investigate the properties of this type of operators in the case that $2^{X}$ is
algebraically complete, i.e., when it is added the inverse elements for the union.

On the other hand, since the discovery of matter-antimatter duality, many mathematicians have been interested in translating this physical phenomenon into Set-Theoretic language. For this reason, we consider this goal in our work. Thus, we have the following scheme:
\begin{center}
\begin{tabular}{c|c|c}
Physic            & $\rightarrow$ & $\sigma$-Set Theory \\
\hline
particle          & $\rightarrow$ & $\sigma$-element \\
antiparticle      & $\rightarrow$ & $\sigma$-antielement \\	
matter            & $\rightarrow$ & $\sigma$-set \\
antimatter        & $\rightarrow$ & $\sigma$-antiset \\
$\xi$ energy      & $\rightarrow$ & $\emptyset$ emptyset.\\
\end{tabular}
\end{center}

In relation to the axiomatic model used in order to build our axiomatic system, we can observe that we first use the axioms of ZF Set Theory. However, after many tries to obtain a consistent axiom system, we realized that this model is not sufficient in order to give a proper existence to the concept of $\sigma$-antielement and $\sigma$-antiset. For this reason, we use the axioms system of NBG Set Theory and MK Set Theory. Then, the generic object of $\sigma$-Set Theory is called $\sigma$-class.

In the present article, we develop an individual analysis of the axioms and we present the results that we consider necessary for the creation of the $\sigma$-antielements and $\sigma$-antisets. Concerning the axiom system of $\sigma$-Set Theory we will have that: The axioms \ref{axiom empty set} (Empty $\sigma$-Set), \ref{axiom of extensionality} (Extensionality), \ref{axiom of creation of class} (Creation of $\sigma$-Class), \ref{axiom of replacement} (Replacement) and \ref{axiom of power set} (Power $\sigma$-Set) are analogous to some axioms of ST. Axioms \ref{axiom of pairs} (Pairs ) and \ref{axiom of fusion} (Fusion) are modifications of the axioms of ``pairs'' and ``union''. Axiom \ref{axiom of w-regularity} (Weak Regularity) is a modification of the axiom of Regularity. Axiom \ref{axiom non bounded set} (non $\epsilon$-Bounded $\sigma$-Set) is an extension of the axiom of infinity. Axioms \ref{axiom of weak choice} (Weak Choice), \ref{axiom e-linear set} ($\epsilon$-Linear $\sigma$-Set) and \ref{axiom one set} (Totally Different $\sigma$-Sets) are necessary to define the concept of $\sigma$-antielement. With axioms \ref{axiom of completeness a} (Completeness A)and \ref{axiom of completeness b} (Completeness B) we give the construction rules of pairs, which will be used to decide when a $\sigma$-element has a $\sigma$-antielement. Axiom \ref{axiom of exclusion} (Exclusion) will be used to define the fusion of $\sigma$-sets, which will help us to define the $\sigma$-antiset. Finally, axiom \ref{axiom of generated set} (Generated $\sigma$-Set) guarantees the existence of generated $\sigma$-set by two $\sigma$-sets from which we obtain the Integer Space $3^{X}$, that in some sense will be bigger than $2^{X}$. Thus, we will see that if we consider this axiom system, then we can build the following concepts:
\begin{itemize}
	\item \textbf{$\sigma$-Antielement}: $x^{\st}$ is a $\sigma$-antielement of $x$ if and only if
	$$\{x\}\cup\{x^{\st}\}=\emptyset.$$
	
	\item \textbf{$\sigma$-Antiset}: $X^{\star}$ is a $\sigma$-antiset of $X$ if and only if
	$$X\cup X^{\star}=\emptyset.$$
	
	\item \textbf{Natural numbers}: $\mathbb{N}=\{1,2,3,4,\ldots\},$ where
	$$1=\{\alpha\}, \ 2=\{\alpha,1\}, \ 3=\{\alpha,1,2\}, \ etc...$$
	
	\item \textbf{$\Theta$-Natural numbers}: $\mathbb{N}_{\Theta}=\{1_{\Theta},2_{\Theta},3_{\Theta},4_{\Theta},\ldots\},$ where
	$$1_{\Theta}=\{\emptyset\}, \ 2_{\Theta}=\{\emptyset,1_{\Theta}\}, \ 3_{\Theta}=\{\emptyset,1_{\Theta},2_{\Theta}\}, \ etc...$$
	
	\item \textbf{Antinatural numbers}: $\mathbb{N}^{\star}=\{1^{\st},2^{\st},3^{\st},4^{\st},\ldots\},$ where
	$$1^{\st}=\{\beta\}, \ 2^{\st}=\{\beta,1^{\st}\}, \ 3^{\st}=\{\beta,1^{\st},2^{\st}\}, \ etc...$$
	
	\item \textbf{Power $\sigma$-set}: $2^{X}=\{x: x\subseteq X\}$.

\

  \item \textbf{Generated space}: $\langle 2^{X},2^{Y}\rangle=\{x\cup y: x\in 2^{X}\wedge y\in 2^{Y}\}.$

\
	
	\item \textbf{Integer space}: $3^{X}=\langle 2^{X},2^{X^{\star}}\rangle=\{x\cup y: x\in 2^{X}\wedge y\in 2^{X^{\star}}\}.$
\end{itemize}

The formal definitions of these concepts we will see in the analysis of axioms. In particular, when we present the axioms \ref{axiom of weak choice}, \ref{axiom e-linear set} and \ref{axiom one set} we will see the relation between the $\sigma$-sets $\alpha$ and $\beta$. In Section 4 (Final Comments) we will see some basic properties of the natural $\sigma$-sets, i.e., natural numbers, antinatural numbers and $\Theta$-natural numbers. Also we present some conjectures which are particularly related with \textbf{loops} and \textbf{lattices}. Finally, in this Section we present a summary of the axiomatic system of $\sigma$-Set Theory.

With respect to the presentation of this work we must stand out that we have continued the introduction to the Set Theory that is presented in the books by K. Devlin \cite{Devlin}, K. Hrbacek and T. Jech \cite{Hrbacek}, C. Ivorra \cite{Ivorra} and T. Jech \cite{Jech} in order to show a formal scheme similar to that presented on Set Theory.

Finally we want to emphasize that the idea of introducing the inverse element, as it is known, is not original. In particular, we have seen that the idea of antiset has already been considered in Mathematics (see, W.D. Blizard \cite{Blizard} (Notes, 1997), Fishburn and I. H. La Valle \cite{Fishburn} (1996) and  V. Pratt \cite{Pratt} (1995)). Also, it is worth noting that after investigating other proposals for the development of antisets, we have found the works of A. Sengupta \cite{Sengupta} (2004) and M. Carroll \cite{Carroll} (2009), who consider the concept of antiset in a similar way to that presented in this theory. That is, based on the idea that there exist two sets A and B such that
$$A\cup B=\emptyset.$$

\end{section}

\begin{section}{Introduction to $\sigma$-Set Theory}

It is clear that one of the keys to the success of a mathematical theory is the consistency of his axiomatic system. For this reason, and after finding some inconsistencies between some axioms of previous versions of this theory, we have seen the need to modify the axiomatic system of the $\sigma$-Set Theory. In this regard, we note that here we present a new axiomatic system based on $\sigma$-classes, which preserves and extends the results of earlier versions, results that are interested for the development of the concepts of $\sigma$-antielement and $\sigma$-antiset. Basically, we take the axioms presented for Zermelo-Fraenkel, Neumann-Bernays-Godel and Morse-Kelley, and modify them according to our requirements. In the development of this work, we tried to follow the steps of the creators of these theories when building our new axiom system, for example, the axioms 3.1, 3.2, 3.3 and 3.4 are copies of existing axioms in the theories of ZF, NBG and MK. 

\

The language of the $\sigma$-Set Theory, $\sigma$-LST, is a formal language with predicates ``$=$" (equality) and ``$\in$" (is a $\sigma$-element of), logical symbols ``$\wedge$" (and), ``$\vee$" (or), ``$\neg$" (not), ``$\exists$" (there exists) and ``$\forall$" (for all), variables $\hat{X},\hat{Y},\hat{Z},...$ and (for convenience) brackets $(\cdot), \ [\cdot]$. We use the following abbreviations
\begin{itemize}
	\item $(\forall \hat{Y}\in \hat{X})(\Psi):=(\forall \hat{Y})(\hat{Y}\in \hat{X}\rightarrow \Psi);$
	\item $(\exists \hat{Y}\in \hat{X})(\Psi):=(\exists \hat{Y})(\hat{Y}\in \hat{X} \wedge \Psi);$
	\item $(\exists! \hat{Y}\in \hat{X})(\Psi):=(\exists! \hat{Y})(\hat{Y}\in \hat{X} \wedge \Psi);$
	\item $cto \hat{X}:=(\exists \hat{Y})(\hat{X}\in \hat{Y}).$
\end{itemize}

In order to refer to a generic object we will use the word $\sigma$-class, i.e., we agree that a formula of type $(\forall \hat{X})(\Phi)$ is read ``All $\sigma$-class $\hat{X}$ satisfies $\Phi$". Also, the formula, $cto \hat{X}$, is read ``$\hat{X}$ is a $\sigma$-set". Then we have defined the $\sigma$-sets as the $\sigma$-classes that belong to at least one other $\sigma$-class.

The difference between $\sigma$-class and $\sigma$-set, will be see reflected in the axiom system. A $\sigma$-class which is not a $\sigma$-set is called a proper $\sigma$-class.

We use Roman letters $x,y,z,X,Y,Z,$ etc. to denote $\sigma$-sets, i.e.,
\begin{itemize}
  \item $(\forall x)(\Phi):=(\forall X)(\Phi):=(\forall \hat{X})(cto\hat{X}\rightarrow\Phi);$
  \item $(\exists x)(\Phi):=(\exists X)(\Phi):=(\exists \hat{X})(cto\hat{X}\wedge\Phi);$
  \item $(\exists! x)(\Phi):=(\exists! X)(\Phi):=(\exists! \hat{X})(cto\hat{X}\wedge\Phi).$
\end{itemize}

Also we use the following abbreviations
\begin{itemize}
  \item $(\hat{X}\notin \hat{Y}):=\neg(\hat{X}\in \hat{Y});$
  \item $(\hat{X}\neq \hat{Y}):=\neg(\hat{X}=\hat{Y});$
  \item $(\Phi\underline{\vee}\Psi):=(\Phi\wedge\neg \Psi)\vee (\neg \Phi \wedge \Psi);$
  \item $(\hat{Y}\subseteq \hat{Z}):=(\forall \hat{X})(\hat{X}\in \hat{Y} \rightarrow \hat{X}\in \hat{Z});$
  \item $(\hat{Y}\subset \hat{Z}):=(\hat{Y}\subseteq \hat{Z} \wedge \hat{Y}\neq \hat{Z})$;
  \item $(\exists ! \hat{X})(\Phi):=(\exists \hat{Y})(\forall \hat{X})(\hat{Y}=\hat{X}\leftrightarrow \Phi );$
  \item $(\hat{X},\ldots,\hat{Y}\in \hat{Z}):=(\hat{X}\in \hat{Z})\wedge\ldots\wedge (\hat{Y}\in \hat{Z});$
  \item $(\hat{X},\ldots,\hat{Y}\notin \hat{Z}):=(\hat{X}\notin \hat{Z})\wedge\ldots\wedge (\hat{Y}\notin \hat{Z});$
  \item $(\forall \hat{X},\hat{Y},\ldots,\hat{Z}):=(\forall \hat{X})(\forall \hat{Y})\ldots(\forall \hat{Z});$
  \item $(\exists \hat{X},\hat{Y},\ldots,\hat{Z}):=(\exists \hat{X})(\exists \hat{Y})\ldots(\exists \hat{Z});$
  \item $\hat{Z}=\{\hat{X},\ldots,\hat{Y}\}:=(\forall \hat{W})(\hat{W}\in \hat{Z} \leftrightarrow \hat{W}=\hat{X}\vee\ldots\vee \hat{W}=\hat{Y}).$
\end{itemize}

\begin{definition}
A formula of $\sigma$-LST is called \textbf{atomic} if it is built according to the following rules:
\begin{enumerate}
	\item $\hat{X} \in \hat{Y}$ is an atomic formula and means that $\hat{X}$ is a $\sigma$-element of $\hat{Y}$;
	\item $\neg\Phi$, $\Phi\rightarrow\Psi$, $(\forall\hat{X})(cto \hat{X}\rightarrow \Phi)$ are atomic formulas if $\Phi$ and $\Psi$ are atomic formulas.
\end{enumerate}
\end{definition}

A formula of $\sigma$-LST is called \textbf{normal} if it is equivalent to an atomic formula. Thus we obtain that
$$(\Phi\wedge\Psi), \ (\Phi\vee\Psi), \ (\neg\Phi), \ (\Phi\rightarrow\Psi), \ (\Phi\leftrightarrow\Psi), \ (\forall \hat{X})(\Phi), \ (\exists \hat{X})(\Phi),$$
are normal formulas if $\Phi$ and $\Psi$ are normal formulas. (In general we use capital Greek letters to denote formulas of $\sigma$-LST). The notions of free and bounded variable are defined as usual. A sentence is a formula with no free variables.

\end{section}

\begin{section}{Axioms and Theorems}

As mentioned in the introduction, one of the inspirations for the development of the $\sigma$-Set Theory is the physic phenomenon of matter-antimatter duality. Then it is clear that inside $\sigma$-ST the matter-antimatter duality will be related to the concepts of $\sigma$-set and $\sigma$-antiset through the following equation:
\begin{equation}\label{eq 1}
A\cup A^{\star}=\emptyset,
\end{equation}
where $A$ is a $\sigma$-set and $A^{\star}$ is the $\sigma$-antiset of $A$. Thus, we deduce from equation (\ref{eq 1}), that the empty $\sigma$-set is directly related to the concept of energy. This fact leads to the following idea: It is a physical fact that when a subatomic particle collides with its respective antiparticle these annihilate each other, producing energy. We suppose that the amount of energy produced by annihilation is directly related with the particles in collision. Then we can establish a relationship particle-energy through a function. Following this same reasoning we can establish a relationship bettwen $\sigma$-set and empty $\sigma$-set through a function. If the $\sigma$-antiset is unique for all $\sigma$-set then we can replace the  equation (\ref{eq 1}) for:
\begin{equation}\label{eq 2}
A\cup A^{\star}=\emptyset_{\xi(A)},
\end{equation}
where $\xi$ will be a multifunctional defined by,
$$\xi:3^{X}\rightarrow \mathbb{R}_{\Theta},$$
where $\mathbb{R}_{\Theta}$ is the $\sigma$-set of $\Theta$-real numbers and $3^{X}$ is the integer space with $A,A^{\star} \in 3^{X}$. Then it is clear that we have two different ways of seeing the empty $\sigma$-set, these are:
\begin{enumerate}
	\item Following the $\sigma$-set properties, that is
	$$A\cup \emptyset_{\xi(B)}=\emptyset_{\xi(C)}\cup A=A$$
	and
	$$\emptyset=\emptyset_{\xi(A)}=\emptyset_{\xi(B)},$$
	for all $A,B,C\in 3^{X}$.
	
	\item Following the energetic analogy. In this case we introduce the following relations
	$$(\forall A,B\in 3^{X})(\emptyset_{\xi(A)}\approx\emptyset_{\xi(B)}\Leftrightarrow \xi(A)=\xi(B)).$$
\end{enumerate}

If $\xi(\emptyset)=0$, we obtain that $\emptyset_{\xi(\emptyset)}=\emptyset_{0}$ and it represents the lowest energy state a system can have. Thus if we define,
$$\emptyset_{\xi(A)}\cup\emptyset_{\xi(B)}=\emptyset_{\xi(A)+\xi(B)},$$
for all $A,B\in 3^{X}$, we obtain that 
$$\emptyset_{0}\cup\emptyset_{\xi(A)}=\emptyset_{\xi(A)},$$
for all $A\in 3^{X}$. 

We note that in this work we will see the concept of integer space $3^{X}$, however the concepts of $\Theta$-real numbers and multifunctional will be seen in future investigations, because we need other mathematical tools in order to develop these concepts. However, we thought it appropriate to mention them because we think that they will be useful when framing the $\sigma$-set equations. 

We present the axiom system of $\sigma$-Set Theory.

\begin{subsection}{The Axiom of Empty $\sigma$-set.}\label{axiom empty set}
There exists a $\sigma$-set which has no $\sigma$-elements, that is
$$(\exists X)(\forall x)(x\notin X).$$

\begin{definition}\label{def 1 empty set}
The $\sigma$-set with no $\sigma$-elements is called \textbf{empty
$\sigma$-set} and is denoted by $\emptyset$.
\end{definition}
\end{subsection}

\begin{subsection}{The Axiom of Extensionality.}\label{axiom of extensionality}
For all $\sigma$-classes $\hat{X}$ and $\hat{Y}$, if $\hat{X}$ and $\hat{Y}$ have the same
$\sigma$-elements, then $\hat{X}$ and $\hat{Y}$ are equal, that is
$$(\forall \hat{X},\hat{Y})[(\forall z)(z\in \hat{X} \leftrightarrow z\in \hat{Y})\rightarrow \hat{X}=\hat{Y}].$$

\begin{theorem}\label{theo 2 empty set is unique}
The empty $\sigma$-set is unique.
\end{theorem}
\end{subsection}

\begin{subsection}{The Axiom of Creation of $\sigma$-Class.}\label{axiom of creation of class}
We consider an atomic formula $\Phi(x)$ (where $\hat{Y}$ is not free). Then there exists the $\sigma$-class $\hat{Y}$ of all $\sigma$-sets that satisfy $\Phi(x)$, that is
$$(\exists\hat{Y})(x\in \hat{Y} \leftrightarrow \Phi(x)),$$
with $\Phi(x)$ being an atomic formula where $\hat{Y}$ is not free.

\begin{theorem}\label{theo 1 creation class}
If $\Phi(x)$ is a normal formula, then there exists the $\sigma$-class $\hat{Y}$ of all $\sigma$-sets that satisfies $\Phi(x)$ and it is unique. That is, if $\Phi(x)$ is a normal formula
$$(\exists ! \hat{Y})(\forall x)(x\in\hat{Y}\leftrightarrow \Phi(x)).$$
\end{theorem}

Therefore, by Theorem \ref{theo 1 creation class} if $\Phi(x)$ is a normal formula, we say that
$$\hat{Y}_{\Phi}=\{x:\Phi(x)\},$$
is the $\sigma$-class of all $\sigma$-sets that satisfies $\Phi(x)$. Other relevant theorem related to the creation of a $\sigma$-class is the following.

\begin{theorem}\label{theo 2 creation class}
If $\Phi(x)$ is a normal formula, then
$$(\forall x)(x\in\hat{Y}_{\Phi}\leftrightarrow \Phi(x)).$$
\end{theorem}

It is clear that if $\Phi(x)$ is a normal formula, then $\hat{Y}_{\Phi}$ is a normal term. Hence, we can define the empty $\sigma$-class and the universal $\sigma$-class as
$$\Theta=\{x:x\neq x\}, \ \ \textbf{U}=\{x:x=x\}.$$

\begin{definition}\label{def inter and union of class}
Let $\hat{X}$ and $\hat{Y}$ be $\sigma$-class. Then
\begin{enumerate}
	\item We define the intersection of $\hat{X}$ and $\hat{Y}$ as
	$$\hat{X}\cap\hat{Y}=\{x: x\in\hat{X}\wedge x\in\hat{Y}\};$$
	\item If $\hat{X}$ and $\hat{Y}$ are proper $\sigma$-class, we define the fusion of $\hat{X}$ and $\hat{Y}$ as
	$$\hat{X}\cup\hat{Y}=\{x: x\in\hat{X}\vee x\in\hat{Y}\}.$$
\end{enumerate}
\end{definition}

The difference between the fusion of $\sigma$-classes and $\sigma$-sets we will be seen in Axiom \ref{axiom of fusion} (Fusion).
\end{subsection}

\begin{subsection}{The Axiom of Scheme of Replacement.}\label{axiom of replacement}
The image of a $\sigma$-set under a normal functional formula $\Phi$ is a
$\sigma$-set.

For each normal formula $\Phi(x,y)$, the following normal formula is an Axiom (of
Replacement):
$$(\forall x)(\exists !y)(\Phi(x,y))\rightarrow (\forall X)(\exists Y)(\forall y)(y\in Y\leftrightarrow (\exists x\in X)(\Phi(x,y))).$$

\begin{theorem}\label{theo 3 schema of separation}(Schema of Separation)
Let $\Phi$ be a normal formula. Then for all $\sigma$-set $X$ there exists a unique $\sigma$-set $Y$ such that $x\in Y$ if and only if $x\in X$ and $\Phi(x)$, that is
$$(\forall X)(\exists ! Y)(\forall x)(x\in Y\leftrightarrow x\in X \wedge\Phi(x)).$$
\end{theorem}

\begin{definition}\label{def 3 inter and diff of sets}
Let $X$ and $Y$ be $\sigma$-sets. We define the following operations
on $\sigma$-sets.
\begin{enumerate}
  \item $X\cap Y=\{x\in X: x\in Y \}.$
  \item $X-Y=\{x\in X: x\notin Y \}.$
\end{enumerate}
\end{definition}

By Theorem \ref{theo 3 schema of separation} (Schema of Separation)
it is clear that $X\cap Y$ and $X-Y$ are $\sigma$-sets.
\end{subsection}

Suppose that we want to add the inverse elements for the union in a standard set theory. Thus it is necessary that for all $A\in 2^{X}$ there exists a $\sigma$-set $A^{\star}$ such that,
$$A\cup A^{\star}=\emptyset.$$
So, if $A$ is a singleton, i.e. $A=\{x\}$, then we should have that $A^{\star}=\{x^{*}\}$. Hence
$$\{x\}\cup \{x^{*}\}=\emptyset.$$
Thus, it is necessary to modify the axiom of pairs.

\begin{subsection}{The Axiom of Pairs.}\label{axiom of pairs}
For all $X$ and $Y$ $\sigma$-sets there exists a $\sigma$-set $Z$, called fusion
of pairs of $X$ and $Y$, that satisfies one and only one of the
following conditions:
\begin{description}
  \item[(a)] $Z$ contains exactly $X$ and $Y$,
  \item[(b)] $Z$ is equal to the empty $\sigma$-set,
\end{description}
that is
$$(\forall X,Y)(\exists Z)(\forall a)[(a\in Z \leftrightarrow a=X \vee a=Y ) \underline{\vee}(a\notin Z)].$$

\begin{notation}\label{not 4 pairs}
Let $X$ and $Y$ be $\sigma$-sets. The fusion of pairs of $X$ and $Y$
will be denoted by $\{X\} \cup \{Y\}$.
\end{notation}

\begin{lemma}\label{lemma 4 fusion of pairs}
If $X$ and $Y$ are $\sigma$-sets, then the fusion of pairs of $X$
and $Y$ is unique.
\end{lemma}

\begin{proof}
Assume that $\{ X\} \cup \{Y\}$ satisfies the condition (a) of Axiom
\ref{axiom of pairs} (Fusion of Pairs) then by Axiom \ref{axiom of
extensionality} (Extensionality) $\{ X\} \cup \{Y\}$ is unique.
Otherwise, if $\{ X\} \cup \{Y\}=\emptyset$, then by Theorem
\ref{theo 2 empty set is unique}, we will have that it is unique.
\end{proof}

\begin{theorem}\label{theo 4 oper pairs is conmutative}
If $X$ and $Y$ are $\sigma$-sets and the fusion of pairs of $X$ and
$Y$ satisfies condition (a) of Axiom \ref{axiom of pairs}, then
$$\{X\} \cup \{Y\}=\{X,Y\}.$$
\end{theorem}

The proof of Theorems \ref{theo 2 empty set is unique}, \ref{theo 1 creation class}, \ref{theo 2 creation class}, \ref{theo 3
schema of separation} and \ref{theo 4 oper pairs is conmutative} are standard in Set Theory so we will not include them.

We notice that in a standard set theory, every set can be constructed from the empty set. On the other hand, if we consider the set of natural numbers $\mathbb{N}$, we have that $1=\{\emptyset\}$. Thus, if we return to the idea of adding inverse elements to the union in the power set $2^{\{1\}}=\{\emptyset,\{1\}\}$, then we must define a set $A^{*}:=\{x^{*}\}$ such that  
$$\{1\}\cup \{x^{*}\}=\emptyset.$$ 
So, it is natural to ask, how we can define the set $x^{*}$?
$$1=\{\emptyset\} \ \ \ \ \ \ x^{*}=\{?\}.$$ 
Thus the problem is to construct the set $x^{*}$. In our case, after many attempts to construct the set $x^{*}$, we conclude that this construction would be easier if we disposed of two different sets $\alpha$ and $\beta$ which can not be constructed from the empty set. Then we can define $1=\{\alpha\}$ and $x^{*}=1^{*}=\{\beta\}$. This idea leads to the inclusion of the definition of sets which are non-constructible from the empty set, from which the concept of $\epsilon$-chain is derived. Then we give the definitions and results that will help us formalize these ideas.

We notice that, given a $\sigma$-class $\hat{X}$, if
there exist $x,y,z,$ $\ldots,u,w$ $\sigma$-sets such that
$$(x\in y)\wedge(y\in z)\wedge\ldots \wedge(u\in w)\wedge(w\in \hat{X}),$$
then it will be possible to introduce the concept of
$\epsilon$-chain of $\hat{X}$. $\epsilon$-Chains will be useful in two different ways:
\begin{enumerate}
  \item To distinguish when a $\sigma$-set cannot be constructed from the empty
$\sigma$-set.
  \item To distinguish when two $\sigma$-sets are totally different.
\end{enumerate}
We use the following abbreviations
\begin{itemize}
  \item $\lch x,y,z\rch:=x\in y\in z:= (x\in y\wedge y\in z);$
  \item $u\inch\lch x,y,z\rch:= (u=x\vee u=y\vee u=z);$
  \item $u\not\inch\lch x,y,z\rch:= (u\neq x\wedge u\neq y\wedge u\neq z);$
  \item $\lch x,y,w\rch\difch\lch a,b,c\rch:= (\exists u,v)[u\inch\lch x,y,w\rch\wedge v\inch\lch a,b,c\rch\wedge u\neq v];$
  \item $\lch x,y,w\rch\totdif\lch a,b,c\rch:= (\forall u,v)[u\inch\lch x,y,w\rch\wedge v\inch\lch a,b,c\rch \rightarrow u\neq v];$
	\item $\lch x,y,z\rch\in CH(\hat{X}):=(x\in y\in z\in \hat{X});$
	\item $\exists \lch x,y,z\rch\in CH(\hat{X}):=(\exists x,y,z)(x\in y\in z\in \hat{X});$
	\item $(\exists \lch x,y,z\rch\in CH(\hat{X}))(\Phi):=(\exists x,y,z)(x\in y\in z\in \hat{X}\wedge\Phi);$
	\item $\forall\lch x,y,z\rch\in CH(\hat{X}):=(\forall x,y,z)(x\in y\in z\in \hat{X});$
	\item $(\forall\lch x,y,z\rch\in CH(\hat{X}))(\Phi):=(\forall x,y,z)(x\in y\in z\in \hat{X}\rightarrow \Phi);$
	\item $\lch x,y,z\rch\in CH_{p}(\hat{X}):=(x\in y\in z\in \hat{X}\wedge x,y\in \hat{X});$
	\item $\exists \lch x,y,z\rch\in CH_{p}(\hat{X}):=(\exists x,y,z)(x\in y\in z\in \hat{X}\wedge x,y\in \hat{X});$
	\item $(\exists \lch x,y,z\rch\in CH_{p}(\hat{X}))(\Phi):=(\exists x,y,z)(x\in y\in z\in \hat{X}\wedge x,y\in \hat{X} \wedge\Phi);$
	\item $\forall\lch x,y,z\rch\in CH_{p}(\hat{X}):=(\forall x,y,z)(x\in y\in z\in \hat{X}\wedge x,y\in \hat{X});$
	\item $(\forall\lch x,y\rch\in CH_{p}(\hat{X}))(\Phi):=(\forall x,y)(x\in y\in \hat{X}\wedge x\in \hat{X}\rightarrow\Phi).$
\end{itemize}
Also we use the abbreviations
$$(\forall\lch x,y,z\rch\in CH(\hat{X}))(\forall\lch a,b,c,d\rch\in CH(\hat{Y}))(\Phi):=\hspace{3.5cm}$$
$$\hspace{2.8cm}(\forall x,y,z,a,b,c,d)(x\in y\in z\in \hat{X}\wedge a\in b\in c\in d\in \hat{Y}\rightarrow \Phi).$$
Note that the above formulas can be extended to more variables.

Regarding $\epsilon$-chains, the next definitions will be relevant.

\begin{definition}\label{def 4 chain and proper chain }
Let $\hat{X}$ be a $\sigma$-class and $x,\ldots,w$ be $\sigma$-sets. We say that:
\begin{description}
  \item[(a)] $\lch x,\ldots,w\rch$ is an \textbf{$\epsilon$-chain of $\hat{X}$} if $\lch x,\ldots,w\rch\in CH(\hat{X})$;
  \item[(b)] $\lch x,\ldots,w\rch$ is a \textbf{proper $\epsilon$-chain of $\hat{X}$} if $\lch x,\ldots,w\rch\in CH_{p}(\hat{X}).$
\end{description}
\end{definition}

\begin{definition}\label{def 4 link}
Let $\hat{X}$ be a $\sigma$-class and $x,\ldots,w$ be $\sigma$-sets such that
$$\lch x,\ldots,w\rch\in CH(\hat{X}).$$
We say that $u$ is a \textbf{link of the $\epsilon$-chain} $\lch x,\ldots,w\rch$ if $u\inch\lch x,\ldots,w\rch$.
\end{definition}

In particular, if there exist $x,\ldots,w$ $\sigma$-sets such that $\lch x,\ldots,w\rch\in CH(\hat{X})$ we will say that the $\sigma$-set $x$ is the \textbf{least link} and the $\sigma$-set $w$ is the \textbf{greatest link} of the $\epsilon$-chain.

Let us introduce the abbreviation
$$u,\ldots,v\inch\lch x,\ldots,w\rch:=u\inch\lch x,\ldots,w\rch\wedge\ldots\wedge v\inch\lch x,\ldots,w\rch.$$

\begin{definition}\label{def 4 chain diff and tot diff}
Let $\hat{X}$ be a $\sigma$-class. If there exist $\lch x,\ldots,w\rch,\lch a,\ldots,c\rch \in CH(\hat{X})$ we say that:

\begin{enumerate}
  \item $\lch x,\ldots,w\rch$ \textbf{is different from} $\lch
  a,\ldots,c\rch$ if and only if
  $$\lch x,\ldots,w\rch\difch\lch a,\ldots,c\rch;$$

  \item $\lch x,\ldots,w\rch$ \textbf{is totally different from} $\lch
a,\ldots,c\rch$ if and only if
$$\lch x,\ldots,w\rch\totdif\lch a,\ldots,c\rch;$$

\item $\lch x,\ldots,w\rch$ \textbf{is an extending $\epsilon$-chain of $\hat{X}$}  if
  the least link is non-empty (i.e. $x\neq\emptyset$).
\end{enumerate}
\end{definition}

Also we will say that two $\epsilon$-chains are disjoint if they are
totally different.

\begin{definition}\label{def 4 totality different}
Let $X$ and $Y$ be nonempty $\sigma$-sets. We will say that $X$ and
$Y$ \textbf{are totally different} ($X\totdif Y$) if any
$\epsilon$-chain of $X$ is disjoint to any $\epsilon$-chain of $Y$. That is
$$(\forall\lch x,\ldots,w\rch\in CH(X))(\forall\lch a,\ldots,c\rch\in CH(Y))(\lch x,\ldots,w\rch\totdif\lch a,\ldots,c\rch).$$
\end{definition}

\begin{theorem}\label{theo 4 x totdif y impl x int y empty }
Let $X$ and $Y$ be nonempty $\sigma$-sets. If $X\totdif Y$, then
\begin{description}
  \item[(a)] $X\cap Y=\emptyset$.
  \item[(b)] $X\notin Y \wedge Y\notin X$.
\end{description}
\end{theorem}

\begin{proof}
\begin{description}
  \item[(a)] Suppose that $X\cap Y\neq\emptyset$ then there exists $a\in X\cap Y$. Therefore $\lch a\rch\in CH_{p}(X)$ and $\lch a\rch\in CH_{p}(Y)$, which is a contradiction because $X\totdif Y$.

  \item[(b)] Suppose that $X\in Y$. Since $X$ and $Y$ are nonempty then there exists a $\sigma$-element $a\in X$, so $\lch a\rch\in
CH(X)$. In the same way, $\lch a,X\rch\in CH(Y)$, which is a contradiction since $X\totdif Y$. The proof that $Y\notin X$ is
analogous.
\end{description}
\end{proof}

As already noted, in order to develop the idea of $\sigma$-antiset we consider it necessary to include the existence of at least two $\sigma$-sets, $\alpha$ and $\beta $, which have the following properties:
\begin{enumerate}
	\item $\alpha$ and $\beta$ are non-constructible from the empty $\sigma$-set,
	\item $\alpha$ and $\beta$ are totally different.
\end{enumerate}
On the other hand, we realized that in order to give a good definition of these $\sigma$-sets, they should be the most simple. Hence, it is necessary to include the linear $\epsilon$-root property. The following definitions formalize these ideas.

\begin{definition}\label{def chain ext and LR property}
Let $X$ be a nonempty $\sigma$-set. Then we say that:
\begin{enumerate}

  \item $X$ \textbf{is non-constructible from the empty $\sigma$-set} if
$$(\forall\lch x,\ldots,w\rch\in CH(X))(x\neq\emptyset).$$
That is $\lch x,\ldots,w\rch$ is an extending $\epsilon$-chain of $X$.

 \item $X$ \textbf{has the linear $\epsilon$-root property} if
$$(\forall \lch x,\ldots,w\rch\in CH(X))(\forall u)(u\inch\lch x,\ldots,w\rch\rightarrow(\exists!y)(y\in u)).$$
\end{enumerate}
\end{definition}

In order to denote the $\sigma$-sets that are non-constructible from
empty $\sigma$-set we will use the following $\sigma$-class:
$$NC(\emptyset)=\{X: X \textrm{ is non-constructible from the } \emptyset\},$$
and in order to denote the $\sigma$-sets that have the linear
$\epsilon$-root property we will use the following $\sigma$-class:
$$LR=\{X: X \textrm{ has the linear $\epsilon$-root property}\}.$$

\begin{theorem}\label{theo 5 lin-root implies non-const-empty }
Let $X$ be a nonempty $\sigma$-set. If $X\in LR$, then $X\in
NC(\emptyset)$.
\end{theorem}

\begin{proof}
Consider $X\in LR$ and $x,\ldots,w$ $\sigma$-sets such that $\lch x,\ldots,w\rch\in CH(X)$. Since $X\in
LR$ then there exists a unique $\sigma$-set $y$ such that $y\in
 x$. Therefore $x\neq\emptyset$. Finally $X\in NC(\emptyset)$.
\end{proof}

\begin{theorem}\label{theo 4 extending chain}
Let $X$ be a $\sigma$-set. If $X\in NC(\emptyset)$, then for all
$\lch x,\ldots,w\rch\in CH(X)$, there exists a nonempty $\sigma$-set
$y$ such that $\lch y,x,\ldots,w\rch\in CH(X)$.
\end{theorem}

\begin{proof}
Let $x,\ldots,w$ $\sigma$-sets such that $\lch x,\ldots,w\rch\in CH(X)$. Since $X$ is non-constructible
from the empty $\sigma$-set then $x\neq\emptyset$. Therefore there
exists $y\in x$, so $\lch y,x,\ldots,w\rch\in CH(X)$. Finally we
obtain that $y\neq\emptyset$, because $X$ is non constructively from
the empty $\sigma$-set.
\end{proof}

\begin{corollary}\label{coro 4 x non contr impl emp notin x}
If $X\in NC(\emptyset)$, then $\emptyset\notin X$.
\end{corollary}

\begin{proof}
This fact is obvious by Theorem \ref{theo 4 extending chain}.
\end{proof}

It is clear that in a standard set theory do not exist sets which have the linear $\epsilon$-root property, by the axiom of regularity. Thus, if we want to include this type of $\sigma$-sets, it is necessary to modify this axiom. 

\end{subsection}

\begin{subsection}{The Axiom of Weak Regularity.}\label{axiom of w-regularity}
For all $\sigma$-set $X$, for all $\lch x,\ldots,w\rch\in CH(X)$ we have that $X\not\inch\lch x,\ldots,w\rch$, that is
$$(\forall X)(\forall\lch x,\ldots,w\rch\in CH(X))(X\not\inch\lch x,\ldots,w\rch).$$

\begin{theorem}\label{theo 5.1 w-regularity}
Let $X$ be a $\sigma$-set. Then $X$ is not a $\sigma$-element of $X$.
\end{theorem}

\begin{proof}
Let $X$ be a $\sigma$-set. It is clear that if $X=\emptyset$ then $X\notin X$. Suppose that $X\neq\emptyset$ and $X\in X$. Then $\lch X\rch\in CH(X)$, which is a contradiction with Axiom \ref{axiom of w-regularity} (Weak Regularity).
\end{proof}

\begin{theorem}\label{theo 5.3 w-regularity}
Let $X$ and $Y$ be $\sigma$-sets. If $X$ is a $\sigma$-element of $Y$ then $Y$ is not a $\sigma$-element of $X$.
\end{theorem}

\begin{proof}
Consider $X$ and $Y$ $\sigma$-sets such that $X\in Y$ and $Y\in X$. Then we can obtain $\lch X,Y\rch\in CH(X)$, which is a contradiction with Axiom \ref{axiom of w-regularity} (Weak Regularity).
\end{proof}

\begin{definition}\label{def singleton}
Let $X$ be a $\sigma$-set. We say that $X$ is a \textbf{singleton} if there exists a unique $\sigma$-set $x$ such that $x\in X$.
\end{definition}

In order to denote the singleton we will use the following $\sigma$-class:
$$SG=\{X: X \textrm{ is a singleton}\}.$$

We will introduce the notion of a $\sigma$-subclass of $\hat{X}$. We say that $\hat{Y}$ is a $\sigma$-subclass of a $\sigma$-class $\hat{X}$ if for all $x\in\hat{Y}$ we have $x\in \hat{X}$ (i.e. $\hat{Y}\subseteq \hat{X}$). It is clear that $\hat{X}=\hat{Y}$ if and only if $\hat{X}\subseteq \hat{Y}$ and $\hat{Y}\subseteq \hat{X}$.

We notice that, if we look at how natural numbers are built we have
$$1=\{\alpha\}, \ 2=\{\alpha,1\}, \ 3=\{\alpha,1,2\}, \ 4=\{\alpha,1,2,3\},\ldots $$
and 
$$1^{\st}=\{\beta\}, \ 2^{\st}=\{\beta,1^{\st}\}, \ 3^{\st}=\{\beta,1^{\st},2^{\st}\}, \ 4^{\st}=\{\beta,1^{\st},2^{\st},3^{\st}\},  \ldots. $$
Hence, 
$$\alpha\in 1, \ \alpha\in 2, \ \alpha\in3, \ \alpha\in4,\ldots $$
and 
$$\beta\in 1^{\st}, \ \beta\in 2^{\st}, \ \beta\in 3^{\st}, \ \beta\in 4^{\st},\ldots. $$
Thus, a common property of 1, 2, 3,..., is that $\alpha$ is a $\sigma$-element. On the other hand we have that a common property of $1^{\st}$, $2^{\st}$, $3^{\st}$,..., is that $\beta$ is a $\sigma$-element. This fact leads us to the following definition.

\begin{definition}\label{def 6 min max}
Let $\hat{X}$ be a $\sigma$-class. We define:
\begin{description}
    \item[(a)] Given $y\in \hat{X}$, we say that \textbf{ $y$ is an $\epsilon$-minimal
    $\sigma$-element of $\hat{X}$} if for all $\lch x,\ldots,w\rch\in CH(y)$ we have that $x,\ldots,w\notin \hat{X}$.

    \item[(b)] The \textbf{$\epsilon$-minimal $\sigma$-subclass of $\hat{X}$} is
$$\min(\hat{X})=\{y\in \hat{X}: y \textrm{ is an $\epsilon$-minimal $\sigma$-element of } \hat{X} \}.$$

    \item[(c)] Given $y\in \hat{X}$, we say that \textbf{ $y$ is an $\epsilon$-maximal
    $\sigma$-element of $\hat{X}$} if for all $z\in \hat{X}$ and for all $\lch x,\ldots,w\rch\in CH(z)$ we have that $y\not\inch\lch x,\ldots,w\rch$.

    \item[(d)] The \textbf{$\epsilon$-maximal $\sigma$-subclass of $\hat{X}$} is
$$\max(\hat{X})=\{y\in \hat{X}: y \textrm{ is an $\epsilon$-maximal $\sigma$-element of } \hat{X} \}.$$

\end{description}
\end{definition}

We see that the concepts of maximum and minimum are related to the concept of bounded set. Following these ideas, in our work the concepts of  $\epsilon$-maximal and $\epsilon$-minimal are related to the concept of infinite $\sigma$-set.

We observe that:

\

\hspace{-0.4cm}$y\in \min(\hat{X})\leftrightarrow(y\in \hat{X})\wedge(\forall \lch x,\ldots,w\rch\in CH(y))(x,\ldots,w\notin \hat{X});$\\
$y\notin \min(\hat{X})\leftrightarrow(y\notin \hat{X})\vee(\exists\lch x,\ldots,w\rch\in CH(y))(x\in \hat{X}\vee\ldots\vee w\in \hat{X});$\\
$y\in \max(\hat{X})\leftrightarrow(y\in \hat{X})\wedge(\forall z\in \hat{X})(\forall\lch x,\ldots,w\rch\in CH(z))(y\not\inch\lch x,\ldots,w\rch);$\\
$y\notin \max(\hat{X})\leftrightarrow(y\notin \hat{X})\vee(\exists z\in \hat{X})(\exists\lch x,\ldots,w\rch\in CH(z))(y\inch\lch x,\ldots,w\rch).$

\

It is clear that if $X$ is a $\sigma$-set then $\min(X)$ and $\max(X)$ are $\sigma$-subsets of $X$.

If we want to include the natural numbers, it is clear that we have to add an axiom which allows the existence of this $\sigma$-set. On the other hand, since in this theory there exist $\sigma$-sets which have the linear $\epsilon$-root property, we must modify the axiom of infinity. Then, we present the following axiom.

\end{subsection}

\begin{subsection}{The Axiom of non $\epsilon$-Bounded $\sigma$-Set.}\label{axiom non bounded set}
There exists a non $\epsilon$-bounded $\sigma$-set, that is
$$(\exists X)(\exists y)[(y\in X)\wedge(\min(X)=\emptyset\vee\max(X)=\emptyset)].$$

\begin{definition}\label{def e-bounded set}
Let $X$ be a nonempty $\sigma$-set. We say that:
\begin{enumerate}
	\item $X$ is lower $\epsilon$-bounded if and only if $\min(X)\neq\emptyset$;
	\item $X$ is upper $\epsilon$-bounded if and only if $\max(X)\neq\emptyset$;
	\item $X$ is $\epsilon$-bounded if and only if $X$ is a lower $\epsilon$-bounded and upper $\epsilon$-bounded.
\end{enumerate}
\end{definition}

Also, we say that a $\sigma$-set $X$ is $\epsilon$-finite if it is $\epsilon$-bounded. We say that a $\sigma$-set $X$ is $\epsilon$-infinite if it is non $\epsilon$-bounded. Hence in order to denote the $\sigma$-class of all $\epsilon$-infinite $\sigma$-sets we will use the following notation:
$$IF=\{X: \min(X)=\emptyset\vee\max(X)=\emptyset\},$$
and in order to denote the $\sigma$-class of all $\epsilon$-finite $\sigma$-sets we will use
$$FN=\{X: \min(X)\neq\emptyset\wedge\max(X)\neq\emptyset\}.$$

\begin{lemma}\label{lemma 6 singleton min}
Let $X$ be a $\sigma$-set. If $X\in SG$, then the following statement holds:
\begin{description}
  \item[(a)] $X$ is $\epsilon$-finite (i.e. $X\in FN$).
  \item[(b)] $\min(X)=\max(X)=X$.
\end{description}
\end{lemma}

\begin{proof} We consider $X=\{x\}$.
\begin{description}
  \item[(a)] Suppose that $\min(X)=\emptyset$. Therefore $x\not\in\min(X)$. Thus there exists $\lch y,\ldots, w\rch\in CH(x)$ such that $y\in X\vee\ldots\vee w\in X$. In consequence there exists $u\in\lch y,\ldots,w\rch$ such that $u\in X$. Since $X=\{x\}$ we have $u=x$. Finally, there exists $\lch y,\ldots, w\rch\in CH(x)$ such that $x\in\lch y,\ldots,w\rch$ which is a contradiction with Axiom \ref{axiom of w-regularity} (Weak Regularity). Hence $min(X)\neq\emptyset$.

  If $\max(X)=\emptyset$ then $x\notin\max(X)$. Therefore, there exist $z\in X$ and $\lch y,\ldots,w\rch\in CH(z)$ such that $x\inch\lch y,\ldots,w\rch$. Since $X=\{x\}$ we have $z=x$. So there exists $\lch y,\ldots,w\rch\in CH(x)$ such that $x\in\lch y,\ldots,w\rch$ which is a contradiction with Axiom \ref{axiom of w-regularity} (Weak Regularity). Thus $max(X)\neq\emptyset$.

  \item[(b)] (b.1) We prove that $\min(X)=X$. It is clear, from Definition \ref{def 6 min max} that $\min(X)\subseteq X$. Since $\min(X)\neq\emptyset$ then there exists $y\in\min(X)$, so $y\in X$. Therefore $y=x$ because $X=\{x\}$. Thus $\min(X)=X$.

(b.2) We prove that $\max(X)=X$. It is clear, from Definition \ref{def 6 min max} that $\max(X)\subseteq X$. Since $\max(X)\neq\emptyset$ then there exists $y\in\max(X)$, so $y\in X$. Therefore $y=x$ because $X=\{x\}$. Thus $\max(X)=X$.

Finally by the Axiom \ref{axiom of extensionality} (Extensionality) we have that $\min(X)=\max(X)=X$.
\end{description}
\end{proof}

So far, we have talked about the $\sigma$-sets $1=\{\alpha\}$ and $1^{\st}=\{\beta\}$ but we have not formalized the axioms that allow the existence of these $\sigma$-sets. One of the first axioms, presented in previous versions of this theory, was as follows:
\begin{itemize}
	\item There exist two $\sigma$-sets X and Y such that X and Y contain a unique $\sigma$-element, X is totally different from Y and they have the linear $\epsilon$-root property, that is
	$$(\exists X,Y)(\exists ! x,y)[(x\in X)\wedge(y\in Y)\wedge(X\totdif Y)\wedge(X,Y\in LR)].$$
\end{itemize}
Thus, using this axiom we define the $\sigma$-sets $1=\{\alpha\}$ and $1^{\st}=\{\beta\}$. However, thanks to some observations we realized that this axiom has the following weaknesses: With respect to the $\sigma$-sets $1$ and $1^{\st}$ are by no means uniquely determined by the properties formulated in the axiom. The reasons are
\begin{enumerate}
	\item that they are not a priori distinguishable from each other and
	\item that they could both be replaced by any member of an $\epsilon$-chain below them.
\end{enumerate}
This problem \ motivated the study of \ set theories of \ Zermelo-Fraenkel, Neumann-Bernays-Godel and Morse-Kelley. Thus the generic object of this theory is called $\sigma$-class, which allows us to modify the structure of the axioms in order to consider the existence of the $\sigma$-sets 1 and $1^{\st}$. In order to solve this problem, we added the axioms (Weak Choice) and (Linear $\sigma$-Set). Also, we changed the axiom (One and One$^{*}$ $\sigma$-Sets). In this way, we can select the $\sigma$-sets $1=\{\alpha\}$ and $1^{\st}=\{\beta\}$. Thus, these objects are uniquely determined by the axiom system.

\end{subsection}

\begin{subsection}{The Axiom of Weak Choice.}\label{axiom of weak choice}
If $\hat{X}$ is a $\sigma$-class of $\sigma$-sets, then we can choose a singleton $Y$ whose unique $\sigma$-element come from $\hat{X}$, that is
$$(\forall\hat{X})(\forall x)(x\in\hat{X}\rightarrow (\exists Y)(Y=\{x\})).$$

This axiom guarantees the existence of $1$ and $1^{\st}$ $\sigma$-sets, which will be useful when building the $\sigma$-antielements and $\sigma$-antisets.
\end{subsection}

\begin{subsection}{The Axiom of $\epsilon$-Linear $\sigma$-set.}\label{axiom e-linear set}
There exists a nonempty $\sigma$-set $X$ which has the linear $\epsilon$-root property, that is
$$(\exists X)(\exists y)(y\in X\wedge X\in LR).$$

\begin{definition}\label{def e-linear set}
Let $X$ be a nonempty $\sigma$-set. Then $X$ is called a \textbf{$\epsilon$-linear $\sigma$-set}, if $X$ has the linear $\epsilon$-root property (i.e. $X\in LR$).
\end{definition}

By Axiom \ref{axiom e-linear set} there exists $X\in LR$. Then the $\sigma$-class $LK(X)$ of all link of $\epsilon$-chains of $X$ do not have an $\epsilon$-minimal element.
$$LK(X)=\{u: u\inch\lch x,\ldots,w\rch \textrm{ for some }\lch x,\ldots,w\rch\in CH(X) \}$$
In fact, suppose that there exists $y\in\min(LK(X))$. Thus $y\in LK(X)$ and for all $\lch x,\ldots,w\rch\in CH(y)$ we have that $x,\ldots,w\notin LK(X)$, in particular for all $z\in y$ we have that $z\notin LK(X)$, which is a contradiction. Hence $\min(LK(X))=\emptyset$. We can consider $LK(X)\in IF$, so $LK(X)$ is a $\epsilon$-infinite $\sigma$-set. The existence of $\sigma$-set $LK(X)$ is guaranteed by the Axiom \ref{axiom non bounded set} (non $\epsilon$-Bounded $\sigma$-Set). Also, we notice that $X\notin LK(X)$ by the Axiom \ref{axiom of w-regularity} (Weak Regularity).

Also, it is important to observe that if $X\in LR$ then we can not declare that $X\in SG$, because if $\lch x,\ldots,w\rch\in CH(X)$ then $X\not\inch\lch x,\ldots,w\rch$ by the Axiom \ref{axiom of w-regularity} (Weak Regularity). In the example \ref{example 7 fusion of pairs} we will see that the $\sigma$-set $1_{\Gamma}$ has the linear $\epsilon$-root property but it is not a singleton. Also, in the same example we will see that the $\sigma$-set $X=\{2\}$ is a singleton but $X$ does not have the linear $\epsilon$-root property. However, if $X\in LR$ and $y\in LK(X)$ then $y\in SG\cap LR$.

\begin{lemma}\label{lemma if X LR then link SG-LR}
Let $X$ be a nonempty $\sigma$-set such that $X\in LR$. If $y\in LK(X)$, then $y\in SG\cap LR$.
\end{lemma}

\begin{proof}
Let $X$ be a nonempty $\sigma$-set such that $X\in LR$. Since $X\neq\emptyset$ then $LK(X)\neq\emptyset$. So we consider $y\in LK(X)$. By Definition \ref{def chain ext and LR property} we have that $y\in SG$. Let $\lch x,\ldots,w\rch\in CH(y)$ and $u\inch\lch x,\ldots,w\rch$. Since $y\in LK(X)$ then $x,\ldots,w\in LK(X)$. Therefore $u\in LK(X)$ in consequence $u\in SG$. So $y\in SG\cap LR$.
\end{proof}

Let $X$ be a $\sigma$-set, $X$ is called $\epsilon$-linear singleton if $X\in SG\cap LR$.

\end{subsection}

\begin{subsection}{The Axiom of totally different $\sigma$-sets.}\label{axiom one set}
For all $\epsilon$-linear singleton, there exists a $\epsilon$-linear singleton $Y$ such that $X$ is totally different from $Y$, that is
$$(\forall X\in SG\cap LR)(\exists Y\in SG\cap LR)(X\totdif Y).$$

Let $X$ be a $\sigma$-set, in order to denote the $\sigma$-class of all $\sigma$-sets totally different to $X$ we will use the following notation:
$$TD(X)=\{Y:Y\totdif X\}.$$

We introduce the concept of One and One$^{\st}$ $\sigma$-sets. We consider the following $\sigma$-class
$$SG\cap LR=\{X:X \textrm{ is a singleton and has the linear $\epsilon$-root property }\}.$$
By Axiom \ref{axiom e-linear set} ($\epsilon$-Linear $\sigma$-Set) it is clear that $SG\cap LR\neq\Theta$. Thus, by the Axiom \ref{axiom of weak choice} (Weak Choice) we can choose a $\sigma$-set $1$ whose unique $\sigma$-element come from the $\sigma$-class $SG\cap LR$. So, we can define the One $\sigma$-set. The unique $\sigma$-element of $1$ is denoted by $\alpha$, hence
$$1=\{\alpha\}.$$
It is clear that $1,\alpha\in SG\cap LR$. Following the same argument, we consider the $\sigma$-class
$$TD(\alpha)=\{Y:Y\totdif\alpha\}.$$
Since $\alpha\in SG\cap LR$ by the Axiom \ref{axiom one set} (Totally Different $\sigma$-Sets) we obtain that $TD(\alpha)\neq\Theta$. Hence by Axiom \ref{axiom of weak choice} (Weak Choice) we choose a $\sigma$-set $1^{\st}$ whose unique $\sigma$-element comes from the $\sigma$-class $TD(\alpha)$. Thus, we can define the One$^{\st}$ $\sigma$-set. The unique $\sigma$-element of $1^{\st}$ is denoted for $\beta$, so
$$1^{\st}=\{\beta\}.$$
It is clear that $1^{\st},\beta\in SG\cap LR$. Hence we obtain that $1,1^{\st},\alpha,\beta\in SG\cap LR$.

We present the basic properties that the $\sigma$-sets $1$ and $1^{\st}$ have,

\begin{itemize}
  \item $1$ and $1^{\st}$ are unique;
	\item $1\totdif 1^{\st}\wedge 1\totdif \beta \wedge \alpha\totdif \beta \wedge \alpha\totdif 1^{\st};$
	\item $\min(\alpha)=\alpha\wedge \min(\beta)=\beta;$
  \item $\min(1)=1\wedge \min(1^{\st})=1^{\st}.$
\end{itemize}

Since fusion of pairs can satisfy the conditions (a) and (b) of Axiom \ref{axiom of pairs} then we must precise when the fusion of pairs satisfies one or the other condition. To this end we give the Axioms of Completeness (A) and (B).

Let us introduce the abbreviations
\begin{itemize}
	\item $\min(X,Y)=|A\wedge B|:=(\min(X)=A\wedge \min(Y)= B);$
	\item $\min(X,Y)\neq|A\wedge B|:=(\min(X)\neq A\wedge \min(Y)\neq B);$
	\item $\min(X,Y)=|A\vee B|:=(\min(X)= A\vee \min(Y)= B);$
	\item $\min(X,Y)\neq|A\vee B|:=(\min(X)\neq A\vee \min(Y)\neq B).$
\end{itemize}

So far, we have included the axioms that we consider necessary in order to build the $\sigma$-sets $1=\{\alpha\}$ and $1^{\st}=\{\beta\}$. Thus we have that $1$ and $1^{\st}$ generate natural $\sigma$-sets where the first natural elements should have the following form:
$$1=\{\alpha\}, \ 2=\{\alpha,1\}, \ 3=\{\alpha,1,2\}, \ 4=\{\alpha,1,2,3\},\ldots $$
and 
$$1^{\st}=\{\beta\}, \ 2^{\st}=\{\beta,1^{\st}\}, \ 3^{\st}=\{\beta,1^{\st},2^{\st}\}, \ 4^{\st}=\{\beta,1^{\st},2^{\st},3^{\st}\},  \ldots. $$
If we take account these $\sigma$-sets, it is expected that 
\begin{center}
\begin{tabular}{llll}
$\{1\}\cup\{1^{\st}\}=\emptyset,$ & $\{1\}\cup\{2^{\st}\}=\{1,2^{\st}\},$ & $\{1\}\cup\{3^{\st}\}=\{1,3^{\st}\},$ & $\ldots $\\
$\{2\}\cup\{2^{\st}\}=\emptyset,$ & $\{2\}\cup\{3^{\st}\}=\{2,3^{\st}\},$ & $\{2\}\cup\{4^{\st}\}=\{2,4^{\st}\},$ & $\ldots $\\
$\{3\}\cup\{3^{\st}\}=\emptyset,$ & $\{3\}\cup\{4^{\st}\}=\{3,4^{\st}\},$ & $\{3\}\cup\{5^{\st}\}=\{3,5^{\st}\},$ & $\ldots $\\
$ \hspace{1.8cm} \vdots$ & $\hspace{1.8cm} \vdots$ & $\hspace{1.8cm} \vdots$ & $\ddots$\\
\end{tabular}
\end{center}

Thus it is necessary to include the axioms that allow these facts. In this sense, it is important to note that these axioms are developed from the construction of the natural $\sigma$-sets $\mathbb{N}$ and $\mathbb{N}^{\st}$. So, the $\sigma$-Set Theory is incomplete, because we have not consider the development of the $\sigma$-sets $\mathbb{R}$ and $\mathbb{R}^{\st}$. However, when we define the $\sigma$-sets $\mathbb{R}$ and $\mathbb{R}^{\st}$, it is sufficient to complete these axioms to obtain for example:
$$\{1/2\}\cup\{(1/2)^{\st}\}=\emptyset, \ \{\sqrt{2}\}\cup\{\sqrt{2}^{\st}\}=\emptyset, \ \{\pi\}\cup\{\pi^{\st}\}=\emptyset.$$
For now we will concentrate on the construction of the natural $\sigma$-sets. 

We define the formula
$$\Psi(z,w,a,x):=(\exists ! w)(\{z\}\cup\{w\}=\emptyset) \wedge (\forall a)(\{z\}\cup\{a\}=\emptyset\rightarrow a\in x).$$

\end{subsection}

\begin{subsection}{The Axiom of Completeness (A).}\label{axiom of completeness a}
If $X$ and $Y$ are $\sigma$-sets, then
$$\{X\}\cup\{Y\}=\{X,Y\},$$
if and only if $X$ and $Y$ satisfy one of the following conditions:
\begin{description}
  \item[(a)] $\min(X,Y)\neq |1\vee 1^{\st}|\wedge \min(X,Y)\neq |1^{\st}\vee 1|.$
  \item[(b)] $\neg(X\totdif Y).$
  \item[(c)] $(\exists z\in X)[z\notin \min(X)\wedge \neg \Psi(z,w,a,Y)].$
  \item[(d)] $(\exists z\in Y)[z\notin \min(Y)\wedge \neg \Psi(z,w,a,X)].$
\end{description}

\begin{lemma}\label{lemma 7 fusion of pairs x to x singleton}
Let $X$ be a $\sigma$-set. Then $\{X\}\cup\{X\}=\{X\}$.
\end{lemma}

\begin{proof}
\
\begin{description}
  \item[(a)] If $X=\emptyset$, then $\min(X)=\emptyset$. Therefore $\min(\emptyset)\neq 1$ and $\min(\emptyset)\neq 1^{\st}$. Hence $X=\emptyset$ satisfies the condition (a) of Axiom \ref{axiom of completeness a} (Completeness A). In consequence $\{\emptyset\}\cup\{\emptyset\}=\{\emptyset\}.$

  \item[(b)] If $X\neq\emptyset$ then it is clear that $X$ is not totally
different from $X$, then $X$ satisfies the condition (b) of Axiom \ref{axiom of completeness a}
(Completeness A). Thus $\{X\}\cup\{X\}=\{X\}$.
\end{description}
\end{proof}
From now on, we will denote by $1_{\Theta}$ the $\sigma$-set whose only element is $\emptyset$.

\begin{theorem}\label{theo 7 fusion of pairs}
If $X$ is a $\sigma$-set, then
\begin{description}
  \item[(a)] $\{\emptyset\}\cup\{X\}=\{\emptyset,X\}.$
  \item[(b)] $\{\alpha\}\cup\{X\}=\{\alpha,X\}.$
  \item[(c)] $\{\beta\}\cup\{X\}=\{\beta,X\}.$
\end{description}
\end{theorem}

\begin{proof}
It is clear that
$$\min(\alpha)\neq 1\wedge \min(\alpha)\neq 1^{\st},$$
$$\min(\beta)\neq 1\wedge \min(\beta)\neq 1^{\st},$$
$$\min(\emptyset)\neq 1\wedge \min(\emptyset)\neq 1^{\st}.$$
Therefore the proofs of $(a)$, $(b)$ and $(c)$ are obvious.
\end{proof}

Let us see some examples where Axiom \ref{axiom of completeness
a} (Completeness A) can be applied.

\begin{example}\label{example 7 fusion of pairs}
\
\begin{enumerate}
  \item $\{\emptyset\}\cup\{\emptyset\}=\{\emptyset\}=1_{\Theta},$
  \item $\{\emptyset\}\cup\{\alpha\}=\{\emptyset,\alpha\}=1_{\Lambda},$
  \item $\{\emptyset\}\cup\{\beta\}=\{\emptyset,\beta\}=1_{\Omega},$
  \item $\{\alpha\}\cup\{\alpha\}=\{\alpha\}=1,$
  \item $\{\alpha\}\cup\{\beta\}=\{\alpha,\beta\}=1_{\Gamma},$
  \item $\{\beta\}\cup\{\beta\}=\{\beta\}=1^{\st},$
  \item $\{\emptyset\}\cup\{1_{\Theta}\}=\{\emptyset,1_{\Theta}\}=2_{\Theta},$
  \item $\{\emptyset\}\cup\{1_{\Lambda}\}=\{\emptyset,1_{\Lambda}\}=2_{(\emptyset,\Lambda)},$
  \item $\{\emptyset\}\cup\{1_{\Omega}\}=\{\emptyset,1_{\Omega}\}=2_{(\emptyset,\Omega)},$
  \item $\{\emptyset\}\cup\{1_{\Gamma}\}=\{\emptyset,1_{\Gamma}\}=2_{(\emptyset,\Gamma)},$
  \item $\{\emptyset\}\cup\{1\}=\{\emptyset,1\}=2_{\Theta},$
  \item $\{\emptyset\}\cup\{1^{\st}\}=\{\emptyset,1^{\st}\}=2^{\st}_{\Theta},$
  \item $\{\alpha\}\cup\{1_{\Theta}\}=\{\alpha,1_{\Theta}\}=2_{(\alpha,\Theta)},$
  \item $\{\alpha\}\cup\{1_{\Lambda}\}=\{\alpha,1_{\Lambda}\}=2_{(\alpha,\Lambda)},$
  \item $\{\alpha\}\cup\{1_{\Omega}\}=\{\alpha,1_{\Omega}\}=2_{(\alpha,\Omega)},$
  \item $\{\alpha\}\cup\{1_{\Gamma}\}=\{\alpha,1_{\Gamma}\}=2_{(\alpha,\Gamma)},$
  \item $\{\alpha\}\cup\{1\}=\{\alpha,1\}=2,$
  \item $\{\alpha\}\cup\{1^{\st}\}=\{\alpha,1^{\st}\}=2^{\st}_{\alpha},$
  \item $\{\beta\}\cup\{1_{\Theta}\}=\{\beta,1_{\Theta}\}=2_{(\beta,\Theta)},$
  \item $\{\beta\}\cup\{1_{\Lambda}\}=\{\beta,1_{\Lambda}\}=2_{(\beta,\Lambda)},$
  \item $\{\beta\}\cup\{1_{\Omega}\}=\{\beta,1_{\Omega}\}=2_{(\beta,\Omega)},$
  \item $\{\beta\}\cup\{1_{\Gamma}\}=\{\beta,1_{\Gamma}\}=2_{(\beta,\Gamma)},$
  \item $\{\beta\}\cup\{1\}=\{\beta,1\}=2_{\beta},$
  \item $\{\beta\}\cup\{1^{\st}\}=\{\beta,1^{\st}\}=2^{\st},$
  \item $\{1\}\cup\{2\}=\{1,2\},$
  \item $\{1^{\st}\}\cup\{2^{\st}\}=\{1^{\st},2^{\st}\},$

\end{enumerate}
\end{example}

Results from (1) to (24) of Example \ref{example 7 fusion of pairs}
follow from Theorem \ref{theo 7 fusion of pairs} and therefore the
condition (a) of Axiom \ref{axiom of completeness a} (Completeness
(A)) holds. Results (25) and (26) follow from condition (b) of Axiom
\ref{axiom of completeness a}.

We present the following schemas. Let us introduce the following notation
\begin{itemize}
	\item $x\rightarrow y:= x\in y.$
\end{itemize}

\begin{center}
\begin{tabular}{ccccccc}
         &               &          &               & $\{2\}$    \\
         &               &          &               & $\uparrow$ \\
         &               &          &               & $2$          & $\rightarrow$ & $\{1,2\}$ \\
         &               &          & $\nearrow$    & $\uparrow$   & $\nearrow$ \\
$\cdots$ & $\rightarrow$ & $\alpha$ & $\rightarrow$ & $1$          & $\rightarrow$ & $\{1,1_{\Gamma}\}$ \\
         &               &          & $\searrow$    &              & $\nearrow$  \\
         &               &          &               & $1_{\Gamma}$ & $\rightarrow$ & $\{1_{\Gamma}\}$ \\
         &               &          & $\nearrow$    &              & $\searrow$   \\
$\cdots$ & $\rightarrow$ & $\beta$  & $\rightarrow$ & $1^{\st}$      & $\rightarrow$ & $\{1^{\st},1_{\Gamma}\}$ \\
         &               &          & $\searrow$    & $\downarrow$ & $\searrow$ \\
         &               &          &               & $2^{\st}$      & $\rightarrow$ & $\{1^{\st},2^{\st}\}$ \\
         &               &          &               & $\downarrow$ \\
         &               &          &               & $\{2^{\st}\}$  \\
\end{tabular}
\end{center}

It is clear that
\begin{itemize}
  \item $\min(1_{\Gamma})=1_{\Gamma},$
  \item $1,1^{\st}\in SG\cap LR$,
	\item $1_{\Gamma}\notin SG$ and $1_{\Gamma}\in LR$,
	\item $\{2\},\{2^{\st}\},\{1_{\Gamma}\}\in SG$ and $\{2\},\{2^{\st}\},\{1_{\Gamma}\}\notin LR$,
	\item $\{1,1_{\Gamma}\},\{1^{\st},1_{\Gamma}\},\{1,2\},\{1^{\st},2^{\st}\}\notin SG$
	\item $\{1,1_{\Gamma}\},\{1^{\st},1_{\Gamma}\},\{1,2\},\{1^{\st},2^{\st}\}\notin LR$.
\end{itemize}

\begin{center}
\begin{tabular}{ccccccccc}
\\
        &             &        &             &$1$           & $\rightarrow$ & $2$ & $\rightarrow$ &$\cdots$\\
        &             &        &$\nearrow$   \\
$\cdots$&$\rightarrow$&$\alpha$&             &              &               & $2_{\Lambda}$ & $\rightarrow$ &$\cdots$\\
        &             &        &$\searrow$   &              & $\nearrow$    \\
        &             &        &             &$1_{\Lambda}$ & $\rightarrow$ & $2_{(\emptyset,\Lambda)}$ & $\rightarrow$ &$\cdots$\\
        &             &        &             &$\uparrow$    & $\nearrow$    \\
        &             &        &             &$\emptyset$   & $\rightarrow$ & $1_{\Theta}$ &$\rightarrow$ & $2_{\Theta}$ \\
        &             &        &             &$\downarrow$  & $\searrow$    \\
        &             &        &             &$1_{\Omega}$  & $\rightarrow$ & $2_{(\emptyset,\Omega)}$ & $\rightarrow$ &$\cdots$\\
        &             &        &$\nearrow$   &              & $\searrow$    \\
$\cdots$&$\rightarrow$&$\beta$ &             &              &               & $2_{\Omega}$ & $\rightarrow$ &$\cdots$\\
        &             &        &$\searrow$   \\
        &             &        &             &$1^{\st}$     & $\rightarrow$ & $2^{\st}$ & $\rightarrow$ &$\cdots$\\

\\
\end{tabular}
\end{center}

\end{subsection}

\begin{subsection}{The Axiom of Completeness (B).}\label{axiom of completeness b}
If $X$ and $Y$ are $\sigma$-sets, then
$$\{X\}\cup\{Y\}=\emptyset,$$
if and only if $X$ and $Y$ satisfy the
following conditions:
\begin{description}
  \item[(a)] $\min(X,Y)=|1\wedge 1^{\st}|\vee \min(X,Y)=|1^{\st}\wedge 1|$;
  \item[(b)] $X\totdif Y;$
  \item[(c)] $(\forall z)(z\in X\wedge z\notin \min(X))\rightarrow \Psi(z,w,a,Y))$;
  \item[(d)] $(\forall z)(z\in Y\wedge z\notin \min(Y))\rightarrow \Psi(z,w,a,X))$.
\end{description}

\begin{definition}\label{def 8 antielement}
Let $X$ and $Y$ be $\sigma$-sets. If $\{X\}\cup\{Y\}=\emptyset$, then $Y$ \textbf{is called
the $\sigma$-antielement of }$X$, and will be denoted by $X^{\st}$.
\end{definition}

We notice that the basic idea for the inclusion of the Axioms \ref{axiom of completeness a},\ref{axiom of completeness b} (Completeness A and B) comes from the following schemas:

\begin{center}
\begin{tabular}{crccccccc}
\\
     &$3 \ $     & $=\{ \alpha, $ & $1,$           & $2 \ \}$\\
(a)  &       &                & $\updownarrow$ & $\updownarrow$ \ \ &             & $\rightarrow$ & $\{3\}\cup\{4^{\st}\}=\{3,4^{\st}\}$ & by AX. \ref{axiom of completeness a} \\
     &$4^{\st}$ & $=\{ \beta,  $ & $1^{\st},$       & $2^{\st},$       & $3^{\st} \}$\\
\\
\end{tabular}
\end{center}

\begin{center}
\begin{tabular}{ccccccccc}
\\
    & $4 \ $     & $=\{ \alpha, $ & $1,$           & $2 \ ,$           & $3 \ \}$\\
(b) &         &                & $\updownarrow$ & $\updownarrow \ \ $  & $\updownarrow \ \ $ & $\rightarrow$ & $\{4\}\cup\{4^{\st}\}=\emptyset$ & \ \ \ \ \ \ \ by AX. \ref{axiom of completeness b} \\
    &$4^{\st}$  & $=\{ \beta,  $ & $1^{\st},$       & $2^{\st},$       & $3^{\st} \}$\\
\\
\end{tabular}
\end{center}

Remember that the axioms 11 and 12 give the rules for the construction of fusion of pairs. However these axioms are constructed thinking in the natural $\sigma$-sets $\mathbb{N}$ and $\mathbb{N}^{\st}$ and are thus incomplete. 

\begin{theorem}\label{theo 8 oper pairs is conmutative}
Let $X$ and $Y$ be $\sigma$-sets. Then $\{X\}\cup\{Y\}=\emptyset$, if and only if $\{Y\}\cup\{X\}=\emptyset$.
\end{theorem}

\begin{proof}
$(\rightarrow)$ Suppose that $\{X\}\cup\{Y\}=\emptyset$, then $X$ and $Y$ satisfies the conditions (a), (b), (c)
and (d) of Axiom \ref{axiom of completeness b} (Completeness B).
In consequence we obtain the following:

\begin{description}
  \item[(a)] Since $\min(X,Y)=|1\wedge 1^{\st}|$
or $\min(X,Y)=|1^{\st}\wedge 1|$, it is clear that $\min(Y,X)=|1^{\st}\wedge 1|$ or $\min(Y,X)=|1\wedge 1^{\st}|$.

  \item[(b)] Since $X\totdif Y$, then $Y\totdif X$ by Definition \ref{def 4 totality different}.

  \item[(c)] Let $z\in Y$ such that $z\notin \min(Y)$. Since $X$ and $Y$ satisfy the condition (d) of Axiom \ref{axiom of completeness b},
there exists a unique $w$ such that $\{z\}\cup\{w\}=\emptyset$ and given $a$ such that $\{z\}\cup\{a\}=\emptyset$ we have $a\in X$. In
consequence condition (c) applies for the fusion of pairs of $Y$ and $X$.

  \item[(d)] Following the previous proof, this condition is satisfied. Therefore $\{Y\}\cup\{X\}=\emptyset$.
\end{description}

$(\leftarrow)$ The converse follows in a similar way.
\end{proof}

\begin{corollary}\label{coro 8 fusion of pairs is commutative}
If $X$ and $Y$ are $\sigma$-sets, then the fusion of pairs of $X$ and $Y$ is commutative.
\end{corollary}

\begin{proof}
This fact is obvious by Theorems \ref{theo 4 oper pairs is conmutative} and \ref{theo 8 oper pairs is conmutative}.
\end{proof}

Regarding $\sigma$-antielements, the next results will be relevant.

By Theorem \ref{theo 7 fusion of pairs}, we obtain that there exist $\sigma$-sets without $\sigma$-antielement as the case of the empty $\sigma$-set. However, we will prove that, if a $\sigma$-set $X$ has a $\sigma$-antielement then it is unique.

\begin{theorem}\label{theo 8 antielement is unique}
Let $X$ be a $\sigma$-set. If there exists $X^{\st}$ the $\sigma$-antielement of $X$, then $X^{\st}$ is unique.
\end{theorem}

\begin{proof}
Consider two $\sigma$-sets $X$ and $X^{\st}$ such that $\{X\}\cup\{X^{\st}\}=\emptyset$. Suppose that there exists $\widehat{X}$, a $\sigma$-set such that $\{X\}\cup\{\widehat{X}\}=\emptyset$ and $X^{\st}\neq \widehat{X}$. Without loss of generality, we may assume  that $\min(X)=1$ and $\min(X^{\st})=1^{\st}$. Since $\min(X)=1$ and $\{X\}\cup\{\widehat{X}\}=\emptyset$ then $\min(\widehat{X})=1^{\st}$. Therefore, it is clear that $\beta\in X^{\st}\cap \widehat{X}$.

As $X^{\st}\neq \widehat{X}$ there exists $c\in \widehat{X}$ such that $c\notin X^{\st}$ or there exists $a\in X^{\st}$ such that $a\notin \widehat{X}$. If $c\in\widehat{X}$ and $c\notin X^{\st}$ then $c\notin \min(\widehat{X})$. Therefore there exists a unique $w$, such that $\{c\}\cup\{w\}=\emptyset$. Also, we obtain that $w\in X$. Since $\{c\}\cup\{w\}=\emptyset$ by Theorems \ref{theo 7 fusion of pairs} and \ref{theo 8 oper pairs is conmutative} we obtain that $w\neq\alpha$. Therefore $w\notin \min(X)$.

On the other hand, since $\{X\}\cup\{X^{\st}\}=\emptyset$, $w\in X$ and $w\notin\min(X)$, we obtain that there exists a unique $v$ such that $\{w\}\cup\{v\}=\emptyset$. So $v\in X^{\st}$. Finally, by Theorem \ref{theo 8 oper pairs is conmutative} we obtain that $\{w\}\cup\{v\}=\{v\}\cup\{w\}=\emptyset$. Thus $c=v\in X^{\st}$, which is a contradiction.

When there exists $a\in X^{\st}$ such that $a\notin \widehat{X}$ and we get the same contradiction. In consequence $X^{\st}$ is unique.
\end{proof}

\begin{corollary}\label{lemma 11.1}
Let $X$ be a $\sigma$-set. If there exists $X^{\st}$, the
$\sigma$-antielement of $X$, then $(X^{\st})^{\st}=X$.
\end{corollary}

\begin{proof}
By Theorem \ref{theo 8 antielement is unique}, it is enough to prove
that $\{X^{\st}\}\cup\{X\}=\emptyset$. But this fact is obvious by
Theorem \ref{theo 8 oper pairs is conmutative}.
\end{proof}

\begin{theorem}\label{theo 8 antielement of 1 is unique}
Consider the $\sigma$-sets $1$ and $1^{\st}$. Then $\{1\}\cup\{1^{\st}\}=\emptyset$.
\end{theorem}

\begin{proof}
It is clear that $\min(1)=1$, $\min(1^{\st})=1^{\st}$ and $1\totdif 1^{\st}$, then the conditions (a) and (b) of Axiom \ref{axiom of completeness b} (Completeness B) holds. The conditions (c) and (d) of Axiom \ref{axiom of completeness b} follow from the fact that $\min(1)=1$ and $\min(1^{\st})=1^{\st}.$
\end{proof}

\begin{theorem}\label{theo 8 2 antielemt of 2 star}
Consider the $\sigma$-sets $2$ and $2^{\st}$. Then $\{2\}\cup \{2^{\st}\}=\emptyset$.
\end{theorem}

\begin{proof}
\begin{description}
  \item[(a)] It is clear that $2=\{\alpha,1\}$ and $2^{\st}=\{\beta,1^{\st}\}$
. Then $\min(2)=1$ and $\min(2^{\st})=1^{\st}$. In consequence the
condition (a) of Axiom \ref{axiom of completeness b} (Completeness
(B)) holds.

  \item[(b)] Suppose that $2$ is not totally different from $2^{\st}$, then there exist $\lch x,\ldots,w \rch\in CH(2)$ and
$\lch a,\ldots,c \rch\in CH(2^{\st})$ that are not disjoint. Therefore there exists $z$ such that $z\inch \lch x,\ldots,w \rch$
and $z\inch \lch a,\ldots,c \rch$. We notice that the greater link of $\epsilon$-chains $\lch x,\ldots,w \rch$ and $\lch a,\ldots,c\rch$ satisfies the following conditions:

  \begin{enumerate}
    \item $w=\alpha\vee w=1$.
    \item $c=\beta\vee c=1^{\st}.$
  \end{enumerate}

  If $z=w$ then $z\neq c$ because $1\totdif 1^{\st}$. Therefore we can construct $\lch z,\ldots,\beta \rch\in CH(1^{\st})$, which is a contradiction
because $1\totdif 1^{\st}$. In the case that $z=c$ we obtain the same contradiction. Finally, if $z\neq w$ and $z\neq c$, then we can
construct $\lch z,\ldots,\alpha \rch\in CH(1)$ and $\lch z,\ldots,\beta\rch\in CH(1^{\st})$, which is a contradiction
because $1\totdif 1^{\st}$. Therefore we obtain that $2\totdif 2^{\st}$.

  \item[(c)] Let $z\in 2$ and $z\notin \min(2)$. Since $z\notin \min(2)$ then $z=1$. By Theorem \ref{theo 8 antielement of 1 is unique} we
have that there exists a unique $1^{\st}$ such that $\{1\}\cup\{1^{\st}\}=\emptyset$. Finally, we consider $a$ such that
$\{1\}\cup\{a\}=\emptyset$, then $a=1^{\st}$, therefore we obtain that $a\in 2^{\st}$.

  \item[(d)] This proof is analogous to the previous one.

\end{description}

\end{proof}

It is important to notice that by Theorem \ref{theo 7 fusion of pairs} there exist $\sigma$-sets which do not have $\sigma$-antielement.

\begin{example}\label{example 8 fusion of pairs}
As a consequence from Theorem \ref{theo 8 antielement of 1 is unique} and the conditions (c) and (d) of Axiom \ref{axiom of completeness a} (Completeness A) we have the following results:

\begin{tabular}{ccccccccc}
\\
    & $1 \ $     & $=\{ \alpha \}$\\
(a) &            &              &             & $\rightarrow$ & $\{1\}\cup\{2^{\st}\}=\{1,2^{\st}\}$ \\
    & $2^{\st}$  & $=\{ \beta,$ & $1^{\st} \}$\\
\\
\end{tabular}

\begin{tabular}{ccccccccc}
\\
    & $2 \ $     & $=\{ \alpha,$ & $1 \}$\\
(b) &            &               &          & $\rightarrow$ & $\{1^{\st}\}\cup\{2\}=\{1^{\st},2\}$ \\
    & $1^{\st}$  & $=\{ \beta \}$\\
\\
\end{tabular}

\begin{tabular}{ccccccccc}
\\
    & $2 \ $     & $=\{ \alpha,$ & $1 \ \}$         \\
(c) &            &               & $\updownarrow \ \ $ &             & $\rightarrow$ & $\{2\}\cup\{3^{\st}\}=\{2,3^{\st}\}$ \\
    & $3^{\st}$  & $=\{ \beta, $ & $1^{\st}, $    & $ 2^{\st} \}.$\\
\\
\end{tabular}
\end{example}

We note that $1$ and $1^{\st}$ are related for the following equation:
$$\{1\}\cup\{1^{\st}\}=\emptyset.$$
Thus $1$ and $1^{\st}$ cannot belong to the same $\sigma$-set. To formalize this fact we have added the following axiom.

\end{subsection}

\begin{subsection}{The Axiom of Exclusion.}\label{axiom of exclusion}
For all $\sigma$-sets $X,Y,Z$, if $Y$ and $Z$ are $\sigma$-elements of $X$ then the fusion of pairs of $Y$ and $Z$ contains exactly $Y$ and $Z$, that is
$$(\forall X,Y,Z)(Y,Z\in X\rightarrow \{Y\}\cup\{Z\}=\{Y,Z\}).$$

\begin{theorem}(\textbf{Exclusion of Inverses})\label{theo 9 exclusion inverses}
Let $X$ be a $\sigma$-set. If $x\in X$, then $x^{\st}\notin X$.
\end{theorem}

\begin{proof}
Assume that $X\neq\emptyset$ and suppose that $x,x^{\st}\in X$, then by Axiom \ref{axiom of exclusion} (Exclusion) we obtain that $\{x\}\cup\{x^{\st}\}\neq\emptyset$, which is a contradiction.
\end{proof}

We notice that, one of the most important characteristics of the $\sigma$-ST is reflected in Theorem \ref{theo 9 exclusion inverses}. On the other hand, it is clear that $1,1^{\st}\in SG\cap LR$, hence by Theorem \ref{theo 9 exclusion inverses} we see that $SG\cap LR$ is a proper $\sigma$-class. Also, if we define the $\sigma$-class
$$AT=\{X: \min(X)=1\vee\min(X)=1^{\st}\},$$
we have that for all $\sigma$-set $A$ such that there exists $A^{\st}$ the $\sigma$-antielement of $A$, then $A,A^{\st}\in AT$, by condition (a) of Axiom \ref{axiom of completeness b} (Completeness B). Therefore, by Theorem \ref{theo 9 exclusion inverses} we see that $AT$ is a proper $\sigma$-class.

\begin{definition}\label{def 9 int and diff antiset}
Let $X$ and $Y$ be $\sigma$-sets. We define two new operations on
$\sigma$-sets:
\begin{enumerate}
  \item $X\capst Y :=\{x\in X: x^{\st}\in Y\};$
  \
  \item $X\difst  Y:=X-(X\capst Y).$
\end{enumerate}
\end{definition}

By Theorem \ref{theo 3 schema of separation}  (Schema of Separation)
it is clear that $X\capst Y$ and $X\difst Y$ are $\sigma$-sets. Also
we observe that
$$X\capst Y\subseteq X \wedge Y\capst X\subseteq Y$$
and
$$X\difst  Y\subseteq X \wedge Y\difst  X\subseteq Y.$$

\begin{example}\label{example 9 star intersection}

By Example \ref{example 7 fusion of pairs}, we have the following:

\begin{tabular}{ccccccccc}
\\
(a) & $2_{\Theta}$ & $=\{ \emptyset,$ & $1_{\Theta} \}$ \\
    &              &                  &               & $\rightarrow$ & $2_{\Theta}\capst 2=\{x\in 2_{\Theta}: x^{\st}\in 2\}=\emptyset.$ \\
    & $2 \ $       & $=\{ \alpha, $   & $1 \}$\\
\\
\end{tabular}

Thus $2_{\Theta}\difst  2= 2_{\Theta}$.

\begin{tabular}{ccccccccc}
\\
 (b)& $2_{\beta}$ & $=\{ \beta,$ & $1 \ \}$        \\
    &             &              & $\updownarrow \ \ $ & $\rightarrow$ & $2_{\beta}\capst 2^{\st}=\{x\in 2_{\beta}: x^{\st}\in 2^{\st}\}=\{1\}.$ \\
    & $2^{\st} $  & $=\{ \beta,$ & $1^{\st} \}$  \\
\\
\end{tabular}

Thus $2_{\beta}\difst  2^{\st}=\{\beta\}$.

\begin{tabular}{ccccccccc}
\\
 (c)& $A \ $      & $=\{ 1,$            & $2 \ \}$          &               & $A\capst A^{\star}=\{x\in A: x^{\st}\in A^{\star}\}=A$ \\
    &             & $ \ \ \ \updownarrow$ & $\updownarrow \ \ $ & $\rightarrow$ & and \\
    & $A^{\star}$ & $=\{ 1^{\st},$      & $2^{\st} \}$      &               & $A^{\star}\capst A=\{x\in A^{\star}: x^{\st}\in A\}=A^{\star}.$\\
\\
\end{tabular}

Thus $A\difst  A^{\star}=\emptyset$ and $A^{\star}\difst A=\emptyset$.
\end{example}

\begin{theorem}\label{theo 9 star intersection}
Let $X$ be a $\sigma$-set $X$. Then
\begin{description}
  \item[(a)] $X\capst \emptyset= \emptyset\capst X=X\capst \alpha= \alpha\capst X=X\capst \beta= \beta\capst X=\emptyset.$
  \item[(b)] $X\capst1_{\Theta}= 1_{\Theta}\capst X=X\capst1= 1\capst X=X\capst1^{\st}= 1^{\st}\capst X=\emptyset.$
  \item[(c)] $X\capst X=\emptyset.$
\end{description}
\end{theorem}

\begin{proof}
\begin{description}
  \item[(a)] It is clear that $X\capst \emptyset= \emptyset\capst X=\emptyset$ by Definition \ref{def 9 int and diff antiset}. Suppose that $X\capst \alpha\neq\emptyset$ or $\alpha\capst X\neq\emptyset$. If $X\capst \alpha\neq\emptyset$, we have that there exist $x\in X$ and $x^{\st}\in \alpha$. By condition (a) of Axiom \ref{axiom of completeness b} (Completeness B), we obtain that $\min(x^{\st})=1$ or $\min(x^{\st})=1^{\st}$. If $\min(x^{\st})=1$ then $\alpha\in x^{\st}$. Therefore $\alpha\in x^{\st}\in \alpha$ which is a contradiction by Axiom \ref{axiom of w-regularity} (Weak Regularity). If $\min(x^{\st})=1^{\st}$ we have $\beta\in x^{\st}$, and therefore $\beta \in x^{\st}\in \alpha$, which is a contradiction because $1\totdif 1^{\st}$. Following a similar reasoning, in the case that $\alpha\capst X\neq\emptyset$ we obtain the same contradiction.

The proof that $X\capst \beta= \beta\capst X=\emptyset$ is analogous to the previous one.

  \item[(b)] This fact is obvious by Theorem \ref{theo 7 fusion of pairs}.

  \item[(c)] This fact is obvious by Theorem \ref{theo 9 exclusion
inverses}.
\end{description}
\end{proof}

\begin{corollary}\label{coro 9 star intersection}
Let $X$ be a $\sigma$-set $X$. Then
\begin{description}
  \item [(a)] $X\difst \emptyset=X$ and $\emptyset\difst X=\emptyset ;$
  \item [(b)] $X\difst \alpha=X$ and $\alpha\difst X=\alpha;$
  \item [(c)] $X\difst \beta=X$ and $\beta\difst X=\beta;$
  \item [(d)] $X\difst 1_{\Theta}=X$ and $1_{\Theta}\difst X=1_{\Theta};$
  \item [(e)] $X\difst 1=X$ and $1\difst X=1;$
  \item [(f)] $X\difst 1^{\st}=X$ and $1^{\st}\difst X=1^{\st};$
  \item [(g)] $X\difst X=X.$
\end{description}
\end{corollary}

\begin{proof}
This proof is obvious by Definition \ref{def 9 int and diff antiset}
and Theorem \ref{theo 9 star intersection}.
\end{proof}

\begin{definition}\label{def 9 set AF}
Let $X$ be a $\sigma$-set. We say that $X$ is \textbf{$\sigma$-antielement free} if for all $A,B\in X$ then $A\capst B=\emptyset.$
\end{definition}

In order to denote a $\sigma$-antielement free $\sigma$-set we will use the following $\sigma$-class:
$$AF=\{X: X \textrm{ is $\sigma$-antielement free }\}.$$

It is clear, by Theorem \ref{theo 9 star intersection}, that the $\sigma$-set $1_{\Gamma}=\{\alpha,\beta\}\in AF$. Nevertheless the $\sigma$-set $X=\{2_{\beta},2^{\st}\}\notin AF$.

\begin{lemma}
Let $X$ be a $\sigma$-set. Then $\{X\}\in AF$.
\end{lemma}

\begin{proof}
This fact is obvious by Theorem \ref{theo 9 star intersection}.
\end{proof}

\end{subsection}

\begin{subsection}{The Axiom of Power $\sigma$-set.}\label{axiom of power set}
For all $\sigma$-set $X$  there exists a $\sigma$-set $Y$, called the power of $X$, whose $\sigma$-elements are exactly the $\sigma$-subsets of $X$, that is
$$(\forall X)(\exists Y)(\forall z)(z\in Y\leftrightarrow z\subseteq X).$$

\begin{definition}\label{def 10 proper subset}
Let $X$ be a $\sigma$-set,
\begin{enumerate}
  \item If $Z\subset X$, then $Z$ is a proper $\sigma$-subset of $X$.
  \item The $\sigma$-set of all $\sigma$-subsets of $X$,
  $$2^{X}=\{z:z\subseteq X\},$$
  is called the power $\sigma$-set of $X$.
\end{enumerate}
\end{definition}

\begin{theorem}\label{theo 10 set AF subset AF}
Let $Z$ be a $\sigma$-set. If $Z\in AF$ and $X\subseteq Z$, then $X\in AF$.
\end{theorem}

\begin{proof}
We consider $Z\in AF$ and $X\subseteq Z$. If $X=\emptyset$ then it is clear that $X\in AF$. Suppose that $X\neq\emptyset$ and $A,B\in X$. Since $X\subseteq Z$ then $A,B\in Z$ and therefore $A\capst B=\emptyset$. So $X\in AF$.
\end{proof}

\begin{theorem}\label{theo 10 set AF}
Let $X$ be a $\sigma$-set. Then $2^{X}\in AF$.
\end{theorem}

\begin{proof}
Suppose that there exist $A,B\in 2^{X}$ such that $A\capst B\neq\emptyset$. Then there exist $x\in A$ and $x^{\st}\in B$. Therefore $x,x^{\st}\in X$, which is a contradiction by Theorem \ref{theo 9 exclusion inverses}. Thus $2^{X}\in AF$.
\end{proof}

\begin{corollary}\label{coro 10 set of singlenton is AF}
Let $X$ be a $\sigma$-set. Then $Z=\{\{x\}:x\in X\}\in AF$.
\end{corollary}

\begin{proof}
This proof is obvious by Theorems \ref{theo 10 set AF subset AF} and \ref{theo 10 set AF}.
\end{proof}

\end{subsection}

\begin{subsection}{The Axiom of Fusion.}\label{axiom of fusion}
For all $\sigma$-sets $X$ and $Y$, there exists a $\sigma$-set $Z$, called fusion of all $\sigma$-elements of $X$ and $Y$, such that $Z$ contains $\sigma$-elements of the $\sigma$-elements of $X$ or $Y$, that is
$$(\forall X,Y)(\exists Z)(\forall b)(b \in Z \rightarrow (\exists z)[(z \in X\vee z\in Y)\wedge(b\in z)]).$$

We define the fusion of $\sigma$-sets.

\begin{definition}\label{def 11 fusion of sets}
Let $X$ and $Y$ be $\sigma$-sets. Then we define the fusion of $X$ and $Y$ as
$$X\cup Y=\{x: (x\in X\difst Y)\vee(x\in Y\difst X)\}.$$
\end{definition}

It is clear, by Definition \ref{def 11 fusion of sets} and Axiom \ref{axiom of extensionality} (Extensionality), that for all
$\sigma$-sets $X$ and $Y$, the fusion of $\sigma$-sets is commutative; that is $X\cup Y=Y\cup X$.

Also, we observe that if $X^{\st}$ is the $\sigma$-antielement of $X$ then $X,X^{\st}\in AT$ where
$$AT=\{X:\min(X)=1\vee\min(X)=1^{\st}\}.$$
Thus, by the Axiom \ref{axiom of weak choice} (Weak Choice) we can choose the singleton $F=\{X^{\st}\}$ and $E=\{X\}$, hence by Axiom \ref{axiom of fusion} (Fusion) there exists $X\cup X^{\st}$, the fusion of $X$ and $X^{\st}$. Remember that $AT$ is a proper $\sigma$-class by Axiom \ref{axiom of exclusion} (Exclusion).

It is important to note the difference between the fusion of $\sigma$-sets and proper $\sigma$-class. Thus
\begin{itemize}
	\item If $\hat{X},\hat{Y}$ are proper $\sigma$-classes, then
	$$\hat{X}\cup\hat{Y}=\{x:x\in\hat{X}\vee x\in\hat{Y}\}.$$
	\item If $X,Y$ are $\sigma$-sets, then
	$$X\cup Y=\{x: x\in X\difst Y\vee x\in Y\difst X\}.$$
\end{itemize}

In general, we are only interested in the properties of the fusion of $\sigma$-sets.

\begin{example}
It is clear that $1,1^{\st},2,2^{\st}\in AT$ where $AT$ is a proper $\sigma$-class. Thus by Axiom \ref{axiom of weak choice} (Weak Choice) we can choose the singleton $A=\{1\}$, $B=\{1^{\st}\}$, $C=\{2\}$ and $D=\{2^{\st}\}$. So by Axiom \ref{axiom of fusion} (Fusion) there exists

\begin{tabular}{llll}
\\
 (a)& $1\cup 1^{\st}$ & $=\{x: x\in 1\difst 1^{\st}\vee x\in 1^{\st}\difst 1\}$ \\
    &                 & $=\{x: x\in 1\vee x\in 1^{\st}\}$ \\
    &                 & $=\{\alpha,\beta\}=\{\alpha\}\cup\{\beta\}=1_{\Gamma}$,\\
\end{tabular}

\begin{tabular}{llll}
\\
 (b)& $2\cup 2^{\st}$ & $=\{x: x\in 2\difst 2^{\st}\vee x\in 2^{\st}\difst 2\}$ \\
    &                 & $=\{x: x\in 1\vee x\in 1^{\st}\}$ \\
    &                 & $=\{\alpha,\beta\}=\{\alpha\}\cup\{\beta\}=1_{\Gamma}$,\\
\end{tabular}

\begin{tabular}{llll}
\\
 (c)& $1\cup 2^{\st}$ & $=\{x: x\in 1\difst 2^{\st}\vee x\in 2^{\st}\difst 1\}$ \\
    &                 & $=\{x: x\in 1\vee x\in 2^{\st}\}$ \\
    &                 & $=\{\alpha,\beta,1^{\st}\}$,\\
\end{tabular}

\begin{tabular}{llll}
\\
 (d)& $\{1\}\cup \{1^{\st}\}$ & $=\{x: x\in \{1\}\difst \{1^{\st}\}\vee x\in \{1^{\st}\}\difst \{1\}\}$ \\
    &                 & $=\{x: x\in \emptyset\vee x\in \emptyset\}$ \\
    &                 & $=\emptyset$,\\
\end{tabular}

\begin{tabular}{llll}
\\
 (e)& $\{1\}\cup \{1^{\st},2^{\st}\}$ & $=\{x: x\in \{1\}\difst \{1^{\st},2^{\st}\}\vee x\in \{1^{\st},2^{\st}\}\difst \{1\}\}$ \\
    &                 & $=\{x: x\in \emptyset\vee x\in \{2^{\st}\}\}$ \\
    &                 & $=\{2^{\st}\}$.\\
\end{tabular}
\end{example}

\begin{theorem}\label{theo 11 fusion of x and empty}
Let $X$ be a $\sigma$-set. Then
\begin{description}
  \item[(a)] $X\cup\emptyset=\emptyset\cup X=X.$
  \item[(b)] $X\cup X=X.$
\end{description}
\end{theorem}

\begin{proof}
\
\begin{description}
  \item[(a)] Since the fusion is commutative, we only prove that $X\cup\emptyset=X$. By Corollary \ref{coro 9 star intersection}, we obtain that $X\difst\emptyset=X$ and $\emptyset\difst X=\emptyset$. Therefore $X\cup\emptyset=\{x:x\in X\}$. By Axiom \ref{axiom of extensionality} (Extensionality), $X\cup\emptyset=X$.

  \item[(b)] It is clear, by Corollary \ref{coro 9 star intersection}, that $X\difst X=X$. Therefore $X\cup X=\{x:x\in X\}$. So by Axiom \ref{axiom of extensionality} (Extensionality) $X\cup X=X$.
\end{description}
\end{proof}

\begin{example}\label{example 11 fusion of sets}
Consider the $\sigma$-sets $X=\{1,2\}$, $Y=\{1^{\st},2^{\st}\}$,
$Z=\{1\}$ and $W=\{\emptyset\}$. Then we obtain the following:
\begin{enumerate}
  \item $X\cup Y=\emptyset.$
  \item $X\cup W=\{1,2,\emptyset\}.$
  \item $X\cup Z=\{1,2\}.$
  \item $Y\cup Z=\{2^{\st}\}.$
\end{enumerate}
\end{example}

Also we can see that the fusion of $\sigma$-sets is not associative. By Theorem \ref{theo 11 fusion of x and empty}  and Example \ref{example 11 fusion of sets} we obtain
$$(Y\cup X)\cup Z= \emptyset \cup Z=Z$$
and
$$Y\cup (X\cup Z)= Y \cup X = \emptyset.$$

It is clear that $Z \neq \emptyset$, so the fusion of $\sigma$-sets is not associative. Thus it is necessary to consider the order on which the $\sigma$-sets are founded, thus we introduce the notion of chain of fusion
\begin{itemize}
  \item $\bigcup\overrightarrow{F}=X\cup Y\cup Z=(X\cup Y)\cup Z,$
  \item $\bigcup\overrightarrow{F}=X\cup Y\cup Z\cup W\cup\ldots=(\ldots(((X\cup Y)\cup Z)\cup W)\cup\ldots).$
\end{itemize}

Thus, given a $\sigma$-set $F$ the fusion of all $\sigma$-elements of $F$, or fusion of $F$, will be denoted for $\bigcup\overrightarrow{F}$.

We observe that, it is important that the fusion in $\sigma$-ST must have the properties of the union in a standard set theory, for example it is necessary that the fusion complies with the laws of Morgan in $2^{X}$. Hence it is important that $2^{X}$ is AF, because in this case the fusion is associative.

\begin{theorem}\label{theo 11 fusion associative set AF}
If $X\in AF$, then the fusion in $X$ is associative, that is
$$(X\in AF)\rightarrow(\forall A,B,C\in X)[(A\cup B)\cup C=A\cup (B\cup C)].$$
\end{theorem}

\begin{proof}
We consider $X\in AF$ and $A,B,C\in X$. Since $X\in AF$, we have that
$$A\capst B=B\capst A=A\capst C=C\capst A=B\capst C=C\capst B=\emptyset.$$

Therefore $A\cup B=\{x: x\in A \vee x\in B\}$ and $B\cup C=\{x: x\in B \vee x\in C\}.$ This fact implies that
$$(A\cup B)\capst C=C\capst(A\cup B)= A\capst (B\cup C)=(B\cup C)\capst A =\emptyset.$$

In fact, suppose that $(A\cup B)\capst C\neq\emptyset$. Then there exists $x\in A\cup B$ such that $x^{\st}\in C$, which is a contradiction. Suppose that $C\capst(A\cup B)\neq\emptyset$, then there exists $x\in C$ such that $x^{\st}\in A\cup B$ which is a contradiction. The proff that $A\capst (B\cup C)=(B\cup C)\capst A =\emptyset$ is analogous. By Definition \ref{def 11 fusion of sets}, we obtain that
$$(A\cup B)\cup C=\{x: x\in A\cup B \vee x\in C\}=\{x: (x\in A \vee x\in B) \vee x\in C\}$$
$$A\cup (B\cup C)=\{x: x\in A \vee x\in B\cup C\}=\{x: x\in A \vee (x\in B\vee x\in C)\}.$$

Finally, by Axiom \ref{axiom of extensionality} (Extensionality), $(A\cup B)\cup C=A\cup (B\cup C).$
\end{proof}

A direct consequence of Theorem \ref{theo 11 fusion associative set AF} is the following: If $X$ is a $\sigma$-set and $A\subseteq X$, then  
$$A=\bigcup_{x\in A}\{x\}.$$

\begin{definition}\label{def Antiset}
Let $X$ and $Y$ be $\sigma$-sets. We say that $Y$ \textbf{is the $\sigma$-antiset of} $X$ if the fusion of $X$ and $Y$ is equal to the empty $\sigma$-set, that is
$$X\cup Y=\emptyset.$$

The $\sigma$-antiset of $X$ will be denoted by $X^{\star}$.
\end{definition}

Consider $X=\{1,2\}$ then $X^{\star}=\{1^{\st},2^{\st}\}$. We observe that the fusion of pairs of $X$ and $X^{\star}$ is nonempty
because $\min(X)=\{1\}$ and $\min(X^{\star})=\{1^{\st}\}$. Therefore $\{X\}\cup\{X^{\star}\}=\{X,X^{\star}\}.$

It is clear by Theorem \ref{theo 7 fusion of pairs}, that there exist $\sigma$-sets without $\sigma$-antiset as the case of the
$\sigma$-set $1_{\Theta}$. However, we will prove that, if a $\sigma$-set $X$ has a $\sigma$-antiset then it is unique.

\begin{lemma}\label{lemma 11.2}
Let $X$ and $Y$ be $\sigma$-sets. If $X\difst Y=\emptyset$ and $x\in X$, then $x^{\st}\in Y$.
\end{lemma}

\begin{proof}
Consider $X$ and $Y$, $\sigma$-sets such that $X\difst Y=\emptyset$. By Definition \ref{def 9 int and diff antiset} we have that $X-(X\capst Y)=\emptyset$. Therefore, if $x\in X$ then $x\in X\capst Y$, so $x^{\st}\in Y$.
\end{proof}

\begin{lemma}\label{lemma 11.3}
Let $X$ and $Y$ be $\sigma$-sets. If $X\cup Y=\emptyset$, then
\begin{description}
  \item[(a)] If $x\in X$, then $x^{\st}\in Y$.
  \item[(b)] If $y\in Y$, then $y^{\st}\in X$.
\end{description}
\end{lemma}

\begin{proof}
Suppose that $X\cup Y=\emptyset$. By Definition \ref{def 11 fusion of sets} we have that
$$X\cup Y=\{x: (x\in X\difst Y)\vee(x\in Y\difst X)\}=\emptyset.$$
Therefore $X\difst Y=\emptyset$ and $Y\difst X=\emptyset$. Finally by Lemma \ref{lemma 11.2} we obtain that if $x\in X$ then $x^{\st}\in Y$, so if $y\in Y$ then $y^{\st}\in X$.
\end{proof}

\begin{theorem}\label{theo 11 antiset is unique}
Let $X$ be a $\sigma$-set. If there exists $X^{\star}$ the $\sigma$-antiset of $X$, then $X^{\star}$ is unique.
\end{theorem}

\begin{proof} Suppose that $X=\emptyset$. By Theorem \ref{theo 11 fusion of x and empty} there exists $X^{\star}=\emptyset$ such that $X\cup X^{\star}=\emptyset$ and it is unique. We consider $X$, a nonempty $\sigma$-set such that there exists $X^{\star}$, the $\sigma$-antiset of $X$. It is clear, by Theorem \ref{theo 11 fusion of x and empty} that $X^{\star}\neq\emptyset$. Suppose that there exists a $\sigma$-set $\widehat{X}$ such that $X\cup\widehat{X}=\emptyset$ and $X^{\star}\neq\widehat{X}$. Therefore by Theorem \ref{theo 11 fusion of x and empty} we have that $\widehat{X}\neq\emptyset$. Since $X^{\star}\neq\widehat{X}$ then there exists $a\in X^{\star}$ such that $a\notin \widehat{X}$ or there exists
$b\in\widehat{X}$ such that $b\notin X^{\star}$. Let $a\in X^{\star}$ and $a\notin \widehat{X}$. Since $a\in X^{\star}$ then by Lemma \ref{lemma
11.3} $a^{\st}\in X$. Therefore $(a^{\st})^{\st}\in \widehat{X}$. Finally, by Corollary \ref{lemma 11.1} we have that $(a^{\st})^{\st}=a$, which is a contradiction. When there exists $b\in\widehat{X}$ such that $b\notin X^{\star}$ we get the same contradiction.
\end{proof}

\begin{corollary}
Let $X$ be a $\sigma$-set. If there exists $X^{\star}$ the $\sigma$-antiset of $X$, then $(X^{\star})^{\star}=X$.
\end{corollary}

\begin{proof}
This fact is obvious by Theorem \ref{theo 11 antiset is unique}.
\end{proof}

We notice that $\emptyset^{\star}=\emptyset$ and $(\emptyset^{\star})^{\star}=\emptyset$. On the other hand, it is important to observe that we can study the properties of $\sigma$-antiset only considering the definition \ref{def Antiset} and assuming the existence of $\sigma$-antielements. Also, if $A$ be a $\sigma$-set and we want to find its $\sigma$-antiset, then we must know what are the $\sigma$-antielements of $\sigma$-elements of $A$. 

Respect to the $\sigma$-antielements is important to notice that the axioms \ref{axiom of completeness a} (Completeness A) and \ref{axiom of completeness b} (Completeness B) are incomplete because these are designed in order to build the natural and antinatural numbers. However, we think that it will be easy to complete these Axioms, in the case that we want to consider the real and antireal numbers, i.e. $\mathbb{R}$ and $\mathbb{R}^{\star}$.

Now we present the definitions and results, which is needed, for the construction of the $\sigma$-sets $\mathbb{N}$ and $\mathbb{N}^{\star}$, from which we get for all $n,m\in\mathbb{N}$ such that $n\neq m$ 
$$\{n\}\cup\{n^{\st}\}=\emptyset$$
and 
$$\{n\}\cup\{m^{\st}\}=\{n,m^{\st}\}.$$

\begin{definition}\label{def successor S(X)}
Let $X$ be a $\sigma$-set. We define the successor of $X$ by
\begin{center}
\begin{tabular}{llll}
\\
$S(X)$ & $= X\cup \{X\}$.\\
       & $=\{x: x\in X\difst \{X\}\vee x\in \{X\}\difst X\}$ \\
\end{tabular}
\end{center}
\end{definition}

\begin{lemma}\label{lemma successor is non empty}
Let $X$ be a $\sigma$-set. Then the following statements holds:
\begin{description}
	\item[(a)] $X\difst \{X\}=X;$
	\item[(b)] $\{X\}\difst X=\{X\}.$
\end{description}
\end{lemma}

\begin{proof}
(a) Suppose that $X=\emptyset$ then it is clear that $\emptyset\difst \{\emptyset\}=\emptyset.$ We consider $X\neq\emptyset$, by Definition \ref{def 9 int and diff antiset}, it is clear that $X\difst \{X\}\subseteq X$. Suppose that there exists $y\in X$ such that $y\notin X\difst \{X\}$. Since $y\in X$ and $y\notin X-X\capst \{X\}$ then $y\in X\capst \{X\}$. Therefore $y\in y^{\st}=X$ which is a contradiction because $y\totdif y^{\st}$.

(b) Suppose that $X=\emptyset$ then it is clear that $\{\emptyset\}\difst\emptyset=\{\emptyset\}.$ We consider $X\neq\emptyset$, by Definition \ref{def 9 int and diff antiset}, it is clear that $\{X\}\difst X\subseteq \{X\}$. Suppose that $X\notin \{X\}\difst X$. Since $X\in \{X\}$ and $X\notin \{X\}-\{X\}\capst X$ then $X\in \{X\}\capst X$. Therefore $X^{\st}\in X$ which is a contradiction because $X\totdif X^{\st}$.
\end{proof}

\begin{corollary}\label{coro 12 successor is non empty}
Let $X$ be a $\sigma$-set. Then $S(X)=\{x: x\in X\vee x\in \{X\}\}$.
\end{corollary}

\begin{proof}
This fact is obvious by Definition \ref{def successor S(X)} and Lemma \ref{lemma successor is non empty}
\end{proof}

Remember that, if $X$ is a $\sigma$-set then

$$y\in \min(X)\leftrightarrow(y\in X)\wedge(\forall \lch x,\ldots,w\rch\in CH(y))(x,\ldots,w\notin X);$$
$$y\notin \min(X)\leftrightarrow(y\notin X)\vee(\exists\lch x,\ldots,w\rch\in CH(y))(x\in X\vee\ldots\vee w\in X).$$

\begin{lemma}\label{lemma 12 min(x) equal to min(s(x))}
Let $X$ be a $\sigma$-set.
\begin{description}
  \item[(a)] If $X=\emptyset$, then $\min(\emptyset)\subset \min(\{\emptyset\})$.
  \item[(b)] If $X\neq\emptyset$ and lower $\epsilon$-bounded, then $\min(X)=\min(S(X))$.
\end{description}
\end{lemma}

\begin{proof}

\

(a) If $X=\emptyset$, then $\min(\emptyset)=\emptyset$. Therefore $\min(\emptyset)\subset \min(\{\emptyset\})=\{\emptyset\}$.

(b) If $X\neq \emptyset$ and $\min(X)\neq\emptyset$ we obtain the following:

\begin{description}
  \item[($\subseteq$)] Suppose that there exists $y\in \min(X)$ such that $y\notin\min(S(X))$. By Corollary \ref{coro 12 successor is non empty} it is clear that $y\in S(X)$, because $y\in \min(X)\subseteq X\subset S(X)$. Since, $y\in \min(X)$ we have that for all $\lch x,\ldots,w\rch\in CH(y)$ then $x,\ldots,w\notin X$. Also, as $y\in S(X)$ and $y\notin\min(S(X))$, there exists $\lch a,\ldots,c\rch\in CH(y)$ such that $a\in S(X)-X\vee\ldots\vee c\in S(X)-X$. Finally there exists $\lch a,\ldots,c,y\rch\in CH(X)$ such that $X\inch\lch a,\ldots,c,y\rch$ which contradicts Axiom \ref{axiom of w-regularity} (Weak Regularity).

  \item[($\supseteq$)] Let $y\in \min(S(X))$. It is clear that $y\in X$ because $X\notin \min(S(X))$. We consider $\lch x,\ldots,w\rch\in CH(y)$. Since $y\in \min(S(X))$ we have that $x,\ldots,w\notin S(X)$. Therefore, by Corollary \ref{coro 12 successor is non empty} we obtain that $x,\ldots,w\notin X$. Thus $y\in \min(X)$.
\end{description}
\end{proof}

\begin{lemma}\label{lemma 12 x totdif y then s(x) totdif s(y)}
Let $X$ and $Y$ be nonempty $\sigma$-sets. If $X\totdif Y$, then $S(X)\totdif S(Y)$.
\end{lemma}

\begin{proof}
We consider $X,Y$ nonempty $\sigma$-sets such that $X\totdif Y$. That is
$$(\forall\lch x\ldots w\rch\in CH(X))(\forall\lch a\ldots c\rch\in CH(Y))(\lch x\ldots w\rch\totdif\lch a\ldots c\rch).$$

Let $\lch x\ldots w\rch\in CH(S(X))$ and $\lch a\ldots c\rch\in CH(S(Y))$. By Corollary \ref{coro 12 successor is non empty}, the respective greater links $w$ and $c$ satisfy the following conditions:
\begin{enumerate}
  \item $w\in X$ or $w=X$.
  \item $c\in Y$ or $c=Y$.
\end{enumerate}
Finally, it is clear that
\begin{itemize}
	\item If $w\in X$ and $c\in Y$, then $\lch x\ldots w\rch\totdif\lch a\ldots c\rch$.
	\item If $w\in X$ and $c=Y$, then $\lch x\ldots w\rch\totdif\lch a\ldots c\rch$.
	\item If $w=X$ and $c\in Y$, then $\lch x\ldots w\rch\totdif\lch a\ldots c\rch$.
	\item If $w=X$ and $c=Y$, then $\lch x\ldots w\rch\totdif\lch a\ldots c\rch$.
\end{itemize}
because $X\totdif Y$. Therefore $S(X)\totdif S(Y)$.
\end{proof}

We introduce the following schemas,

\begin{tabular}{cccccccccccccc}
\\
$\cdots$&$\rightarrow$&$\alpha$&$\rightarrow$& $1$          & $\rightarrow$ & $2$ & $\rightarrow$ & $3$ & $\rightarrow$ & $4$ & $\rightarrow$ & $\cdots$ \\

        &             &        &$\searrow$   \\
        &             &        &             &$1_{\Lambda}$ & $\rightarrow$ & $2_{\Lambda}$ & $\rightarrow$ & $3_{\Lambda}$ & $\rightarrow$ & $4_{\Lambda}$ & $\rightarrow$ &$\cdots$\\

        &             &        &$\nearrow$   \\
        &             & $\emptyset$ &$\rightarrow$ &$1_{\Theta}$ &$\rightarrow$ & $2_{\Theta}$ &$\rightarrow$ &$3_{\Theta}$ &$\rightarrow$ & $4_{\Theta}$ & $\rightarrow$ &$\cdots$\\
        &             &        &$\searrow$   \\

        &             &        &             &$1_{\Omega}$  & $\rightarrow$ & $2_{\Omega}$ & $\rightarrow$ & $3_{\Omega}$  & $\rightarrow$ & $4_{\Omega}$ & $\rightarrow$ & $\cdots$\\

        &             &        &$\nearrow$   \\
$\cdots$&$\rightarrow$&$\beta$ &$\rightarrow$&$1^{\st}$      & $\rightarrow$ & $2^{\st}$ & $\rightarrow$ &$3^{\st}$ & $\rightarrow$ & $4^{\st}$ & $\rightarrow$ & $\cdots$\\
\\
\end{tabular}

\begin{center}
\begin{tabular}{cccccccccccccc}
\\
$\cdots$&$\rightarrow$&$\alpha$\\
        &             &        &$\searrow$   \\
        &             &        &             &$1_{\Gamma}$  & $\rightarrow$ & $2_{\Gamma}$ & $\rightarrow$ & $3_{\Gamma}$  & $\rightarrow$ & $4_{\Gamma}$ & $\rightarrow$ & $\cdots$\\
        &             &        &$\nearrow$   \\
$\cdots$&$\rightarrow$&$\beta$ \\
\\
\end{tabular}
\end{center}

\begin{center}
\begin{tabular}{cccccccccccccc}
\\
$\cdots$&$\rightarrow$&$\alpha$\\
        &             &           &$\searrow$   \\
        &             &$\emptyset$&$\rightarrow$&$1_{\Pi}$  & $\rightarrow$ & $2_{\Pi}$ & $\rightarrow$ & $3_{\Pi}$  & $\rightarrow$ & $4_{\Pi}$ & $\rightarrow$ & $\cdots$\\
        &             &           &$\nearrow$   \\
$\cdots$&$\rightarrow$&$\beta$ \\
\\
\end{tabular}
\end{center}

where we have that
\begin{itemize}
	\item $1=\{\alpha\}$, $2=\{\alpha,1\}$, $3=\{\alpha,1,2\}$, $4=\{\alpha,1,2,3\}$, $\ldots$
	\item $1_{\Pi}=\{\emptyset,\alpha,\beta\}$, $2_{\Pi}=\{\emptyset,\alpha,\beta,1_{\Pi}\}$, $3_{\Pi}=\{\emptyset,\alpha,\beta,1_{\Pi},2_{\Pi}\}$, $\ldots$
	\item $1_{\Lambda}=\{\emptyset,\alpha\}$, $2_{\Lambda}=\{\emptyset,\alpha,1_{\Lambda}\}$, $3_{\Lambda}=\{\emptyset,\alpha,1_{\Lambda},2_{\Lambda}\}$, $\ldots$
	\item $1_{\Theta}=\{\emptyset\}$, $2_{\Theta}=\{\emptyset,1_{\Theta}\}$, $3_{\Theta}=\{\emptyset,1_{\Theta},2_{\Theta}\}$, $\ldots$
	\item $1_{\Omega}=\{\emptyset,\beta\}$, $2_{\Omega}=\{\emptyset,\beta,1_{\Omega}\}$, $3_{\Omega}=\{\emptyset,\beta,1_{\Omega},2_{\Omega}\}$, $\ldots$
	\item $1_{\Gamma}=\{\alpha,\beta\}$, $2_{\Gamma}=\{\alpha,\beta,1_{\Gamma}\}$, $3_{\Gamma}=\{\alpha,\beta,1_{\Gamma},2_{\Gamma}\}$, $\ldots$
	\item $1^{\st}=\{\beta\}$, $2^{\st}=\{\beta,1^{\st}\}$, $3^{\st}=\{\beta,1^{\st},2^{\st}\}$, $4^{\st}=\{\beta,1^{\st},2^{\st},3^{\st}\}$, $\ldots$
\end{itemize}
We can see these $\sigma$-sets are related in the following way:
\begin{center}
\begin{tabular}{rrcccccll}
    &            &               &            & $1_{\Theta}$\\
    &            &               &            & $\uparrow$ \\
    &            &               &            & $\emptyset$ \\  
    &            &               & $\swarrow$ & $\downarrow$  & $\searrow$ \\  
    &            & $1_{\Lambda}$ &            & $1_{\Pi}$     &            & $1_{\Omega}$ \\  
    &            & $\uparrow$    & $\nearrow$ &               & $\nwarrow$ & $\uparrow$ \\ 
    &            & $\alpha$      &            &               &            & $\beta$ \\ 
    & $\swarrow$ &               & $\searrow$ &               & $\swarrow$ &         & $\searrow$ \\ 
$1$ &            &               &            & $1_{\Gamma}$  &            &         &            & $1^{*}$ \\ 
\end{tabular}
\end{center}

As we have noted earlier the axioms \ref{axiom of completeness a} (Completeness A) and \ref{axiom of completeness b} (Completeness B) gives the construction rules of $\sigma$-antielements, which serve to explicitly find the $\sigma$-antisets. Also, it is important to note that these axioms are constructed such that 
\begin{center}
\begin{tabular}{llll}
$\{1\}$ & $\cup$ & $\{1^{\st}\}$ & $=\emptyset$\\
$\{S(1)\}$ & $\cup$ & $\{S(1^{\st})\}$ & $=\emptyset$\\
$\{S(S(1))\}$ & $\cup$ & $\{S(S(1^{\st}))\}$ & $=\emptyset$\\
$\hspace{0.7cm}\vdots$ & $\hspace{0.2cm}\vdots$ & $\hspace{0.7cm}\vdots$ & $= \emptyset$\\
\end{tabular}
\end{center}

On the other hand, if we follow the above schema, we can think in to complete the axioms \ref{axiom of completeness a} (Completeness A) and \ref{axiom of completeness b} (Completeness B) such that 
$$\{1_{\Lambda}\}\cup\{1_{\Omega}\}=\emptyset,$$
$$\{1_{\Pi}\}\cup\{1_{\Gamma}\}=\emptyset $$
and
$$\{1_{\Theta}\}\cup\{x\}=\{1_{\Theta},x\},$$
with $x=1,1^{\st},1_{\Lambda},1_{\Omega},1_{\Pi},1_{\Gamma}$. In future investigations we will study these changes.

Following with our investigation, we give the formal definitions for the construction of inductive $\sigma$-sets.

\begin{definition}\label{def inductive sets}
Let $I$ be a $\sigma$-set,
\begin{description}

  \item[(a)] $I$ is called inductive if:
  \begin{enumerate}
    \item $\min(I)\neq\emptyset$;
    \item If $x\in I$, then $S(x)\in I$.
  \end{enumerate}
\begin{tabular}{cccccccccccccc}
\\
$1$ & $\in$ & $2$ & $\in$ & $3$  & $\in$ & $4$ & $\in$ & $\cdots$  \\
$1_{\Lambda}$  & $\in$ & $2_{\Lambda}$ & $\in$ &  $3_{\Lambda}$  & $\in$ & $4_{\Lambda}$ & $\in$ &$\cdots$\\
$1_{\Theta}$ & $\in$ & $2_{\Theta}$ & $\in$ & $3_{\Theta}$ & $\in$ & $4_{\Theta}$ & $\in$ & $\cdots$  \\
\\
\end{tabular}

  \item[(b)] $I$ is called $\alpha$-inductive if:
  \begin{enumerate}
    \item $\min(I)=\{1\}$;
    \item If $x\in I$ and $x\neq 1$, then $1\in x$.
    \item If $x\in I$, then $S(x)\in I$.
  \end{enumerate}

  In this case we denote the $\alpha$-inductive
  $\sigma$-set by $I_{\alpha}$.

\begin{tabular}{cccccccccccccc}
\\
$1$ & $\in$ & $2$ & $\in$ & $3$  & $\in$ & $4$ & $\in$ & $\cdots$  \\
\\
\end{tabular}

  \item[(c)] $I$ is called $\beta$-inductive if:
  \begin{enumerate}
    \item $\min(I)=\{1^{\st}\}$;
    \item If $x\in I$ and $x\neq 1^{\st}$, then $1^{\st}\in x$.
    \item If $x\in I$, then $S(x)\in I$.
  \end{enumerate}

  In this case we denote the $\beta$-inductive
  $\sigma$-set by $I_{\beta}$.

\begin{tabular}{cccccccccccccc}
\\
$1^{\st}$ & $\in$ & $2^{\st}$ & $\in$ & $3^{\st}$ & $\in$ & $4^{\st}$ & $\in$ & $\cdots$ \\
\\
\end{tabular}

  \item[(d)] $I$ is called $\Theta$-inductive if:
  \begin{enumerate}
    \item $\min(I)=\{1_{\Theta}\}$;
    \item If $x\in I$ and $x\neq 1_{\Theta}$, then $1_{\Theta}\in x$.
    \item If $x\in I$, then $S(x)\in I$.
  \end{enumerate}

  In this case we denote the $\Theta$-inductive
  $\sigma$-set by $I_{\Theta}$.

\begin{tabular}{cccccccccccccc}
\\
$1_{\Theta}$ & $\in$ & $2_{\Theta}$ & $\in$ & $3_{\Theta}$ & $\in$ & $4_{\Theta}$ & $\in$ & $\cdots$ \\
\\
\end{tabular}

\end{description}
\end{definition}

We notice that if $I$ is $\Theta$-inductive, $\alpha$-inductive or $\beta$-inductive, then it is inductive.

Remember that, if $X$ is a $\sigma$-set then
$$y\in \max(X)\leftrightarrow(y\in X)\wedge(\forall z\in X)(\forall\lch x,\ldots,w\rch\in CH(z))(y\not\inch\lch x,\ldots,w\rch);$$
$$y\notin \max(X)\leftrightarrow(y\notin X)\vee(\exists z\in X)(\exists\lch x,\ldots,w\rch\in CH(z))(y\inch\lch x,\ldots,w\rch).$$

\begin{theorem}\label{theo inductive set is non upper bounded}
Let $I$ be an inductive $\sigma$-set. Then $I$ is a non upper $\epsilon$-bounded $\sigma$-set.
\end{theorem}

\begin{proof}
It is clear that $I$ is lower $\epsilon$-bounded by Definition \ref{def inductive sets}. Suppose that there exists $y\in\max(I)$. Therefore, for all $z\in I$ and for all $\lch x,\ldots,w\rch\in CH(z)$ we have that $y\not\inch\lch x,\ldots,w\rch$.

On the other hand, since $y\in I$ and $I$ is an inductive $\sigma$-set, then $S(y)\in I$. Therefore there exists $S(y)\in I$ such that $y\in S(y)$ which is a contradiction. Thus $\max(I)=\emptyset$.
\end{proof}

The existence of inductive $\sigma$-sets is guaranteed by Axiom \ref{axiom non bounded set} (non $\epsilon$-Bounded $\sigma$-Set).

We introduce the concept of $\sigma$-set of all natural numbers, anti-natural numbers and $\Theta$-natural numbers.

\begin{definition}\label{def 12 natural numbers}
Let $X_{\alpha}$, $X_{\beta}$ and $X_{\Theta}$ be $\alpha,\beta$ or $\Theta$-inductive $\sigma$-sets. Then we define:
\begin{description}
  \item[(a)] The $\sigma$-set of all natural numbers as:
$$\mathbb{N}=\{x\in X_{\alpha}: (x\in I_{\alpha})(\forall I_{\alpha})\}.$$

  \item[(b)] The $\sigma$-set of all anti-natural numbers as:
$$\mathbb{N}^{\star}=\{x\in X_{\beta}: (x\in I_{\beta})(\forall I_{\beta})\}.$$

  \item[(c)] The $\sigma$-set of all $\Theta$-natural numbers as:
$$\mathbb{N}_{\Theta}=\{x\in X_{\Theta}: (x\in I_{\Theta})(\forall I_{\Theta})\}.$$
\end{description}
\end{definition}

\begin{theorem}\label{theo natural numbers are inductive}
The following statements holds:
\begin{description}
  \item[(a)] $\mathbb{N}$ is $\alpha$-inductive and if $I_{\alpha}$ is any $\alpha$-inductive $\sigma$-set, then $\mathbb{N}\subseteq I_{\alpha}$;
  \item[(b)] $\mathbb{N}^{\star}$ is $\beta$-inductive and if $I_{\beta}$ is any $\beta$-inductive $\sigma$-set, then $\mathbb{N}^{\star}\subseteq I_{\beta}$;
  \item[(c)] $\mathbb{N}_{\Theta}$ is $\Theta$-inductive and if $I_{\Theta}$ is any $\Theta$-inductive $\sigma$-set, then $\mathbb{N}_{\Theta}\subseteq I_{\Theta}$.
\end{description}
\end{theorem}

\begin{proof}
We will only prove (a) because the  proofs of (b) and (c) are similar.

\begin{itemize}
	\item If $x\in\mathbb{N}$, then $x\in I_{\alpha}$ for all $I_{\alpha}$. Therefore if $x\neq 1$ then $1\in x$.
	\item If $x\in\mathbb{N}$, then $x\in I_{\alpha}$ for all $I_{\alpha}$, so $S(x)\in I_{\alpha}$ for all $I_{\alpha}$. Therefore $S(x)\in\mathbb{N}$.
	\item It is clear that $1\in\mathbb{N}$ because $1\in I_{\alpha}$ for any $I_{\alpha}$. We will prove that $\min(\mathbb{N})=\{1\}$. It is clear by definition that
$$y\in \min(I)\leftrightarrow(y\in I)\wedge(\forall \lch x,\ldots,w\rch\in CH(y))(x,\ldots,w\notin I).$$

Therefore we obtain that $1\in \min(\mathbb{N})$ because $1\in \mathbb{N}$ and for all $\lch x,\ldots,w\rch\in CH(1)$ we have that $x,\ldots,w\notin I_{\alpha}$ for all $I_{\alpha}$. So $x,\ldots,w\notin\mathbb{N}$. Hence $\{1\}\subseteq \min(\mathbb{N})$.

Suppose that $\min(\mathbb{N})\not\subseteq\{1\}$. Then there exists $y\in \min(\mathbb{N})$ such that $y\neq 1$. Since $y\in \min(\mathbb{N})$ then for all $\lch x,\ldots,w\rch\in CH(y)$ we have that $x,\ldots,w\notin\mathbb{N}$. Therefore, we obtain $y\in\mathbb{N}$ such that $y\neq 1$ and for all $\lch x,\ldots,w\rch\in CH(y)$ we have that $x,\ldots,w\notin\mathbb{N}$, which is a contradiction because $1\in y$ and $1\in\mathbb{N}$. Therefore $\min(\mathbb{N})=\{1\}$.
\end{itemize}

The second part of the Theorem \ref{theo natural numbers are inductive} (a) follows immediately from the definition of $\mathbb{N}$.
\end{proof}

We introduce the Principle of Induction for the study of the different types of natural numbers $\mathbb{N}$, $\mathbb{N}^{\star}$
and $\mathbb{N}_{\Theta}$.

\begin{theorem}\label{theo 12 induction principle}{\textbf{The Principle of Induction.}}
Let $\Phi(x)$ be a normal formula.
\begin{enumerate}
  \item $[\Phi(1)\wedge (\forall n\in \mathbb{N})(\Phi(n)\rightarrow \Phi(n\cup \{n\}))]\rightarrow (\forall n\in \mathbb{N})(\Phi(n))$.

  \item $[\Phi(1^{\st})\wedge (\forall n\in \mathbb{N}^{\star})(\Phi(n)\rightarrow \Phi(n\cup \{n\}))]\rightarrow (\forall n\in \mathbb{N}^{\star})(\Phi(n))$.

  \item $[\Phi(1_{\Theta})\wedge (\forall n\in \mathbb{N}_{\Theta})(\Phi(n)\rightarrow \Phi(n\cup \{n\}))]\rightarrow (\forall n\in \mathbb{N}_{\Theta})(\Phi(n))$.

\end{enumerate}
\end{theorem}

\begin{proof}
We will only prove (a) because the proofs of (b) and (c) are similar. Let $\Phi(x)$ a normal formula and $A=\{n\in \mathbb{N}: \Phi(n)\}$. By Theorem \ref{theo 3 schema of separation} it is clear that $A$ is a $\sigma$-set. We prove that $A$ is $\alpha$-inductive.

If $n\in A$ then $n\in\mathbb{N}$. Thus if $n\neq 1$ then $1\in n$.

If $n\in A$ then $\Phi(n)$ holds. Hence $S(n)\in\mathbb{N}$ and $\Phi(S(n))$ holds, which implies that $S(n)\in A$.

It is clear that $A\subseteq\mathbb{N}$ and $1\in A$. Let $\lch x,\ldots,w \rch\inch CH(1)$, then $x,\ldots,w\notin \mathbb{N}$ because $1\in\min(\mathbb{N})$. Therefore $x,\ldots,w\notin A$, so $1\in\min(A)$. Suppose that $\min(A)\not\subseteq\{1\}$. Then there exists $y\in \min(A)$ such that $y\neq 1$. Then we obtain a $\sigma$-element $y\in A$ such that $y\neq 1$. Hence $1\in y$, which is a contradiction. Thus $\min(A)=\{1\}$, thus $A$ is $\alpha$-inductive. Finally $\mathbb{N}= A$.
\end{proof}

\begin{lemma}\label{lemma 12 min(n) equal to min(n+1)}
The following statements holds:
\begin{description}
  \item[(a)] For all $n\in \mathbb{N}$, $\min(n)=1$,
  \item[(b)] For all $n\in \mathbb{N}^{\star}$, $\min(n)=1^{\st}$,
  \item[(c)] For all $n\in \mathbb{N}_{\Theta}$, $\min(n)=1_{\Theta}$.
\end{description}
\end{lemma}

\begin{proof}
We will only prove (a) because the proofs of (b) and (c) are similar. It is clear that $\min(1)=1$. Suppose that $\min(n)=1$, then, by Lemma \ref{lemma 12 min(x) equal to min(s(x))}, we have that $\min(S(n))=1$. Finally, by Theorem \ref{theo 12 induction principle}, we obtain that $\min(n)=1$ for all $n\in \mathbb{N}$.
\end{proof}

\begin{theorem}\label{theo 12 n exists a unique m star antiel}
For all $n\in \mathbb{N}$ there exists a unique $\sigma$-set $m$ such that $\{n\}\cup\{m\}=\emptyset$.
\end{theorem}

\begin{proof}
If $n=1$ then, by Theorem \ref{theo 8 antielement of 1 is unique}, there exists a unique $\sigma$-set $1^{\st}$ such that $\{1\}\cup\{1^{\st}\}=\emptyset$. Suppose that given $n\in\mathbb{N}$ there exists a unique $\sigma$-set $m$ such that $\{n\}\cup\{m\}=\emptyset$. We will prove that for $S(n)$ there exists a unique $\sigma$-set $\widehat{m}$ such that $\{S(n)\}\cup\{\widehat{m}\}=\emptyset$.

\textbf{Existence}: Consider $\widehat{m}= S(m)$. By Lemmas \ref{lemma 12 x totdif y then s(x) totdif s(y)} and \ref{lemma 12 min(n) equal to min(n+1)}, we have that conditions (a) and (b) of Axiom \ref{axiom of completeness b} (Completeness B) are satisfied by $S(n)$ and $S(m)$. We only need to prove that conditions (c) and (d) of Axiom \ref{axiom of completeness b} are satisfied by them. Nevertheless, we only prove condition (c) because the proof of condition (d) is analogous.

Let $z\in S(n)$ such that $z\notin \min(S(n))=\min(n)=1$. By Corollary \ref{coro 12 successor is non empty}, we have that $z\in n$ or $z=n$.
\begin{description}
  \item[(case 1)] Suppose that $z\in n$.

\begin{description}
  \item[(a.1)] Since $z\in n$ and $z\notin \min(S(n))=\min(n)$, then we have that there exists a unique $w$ such that $\{z\}\cup\{w\}=\emptyset$, because $n$ and $m$ satisfy the condition (c) of Axiom \ref{axiom of completeness b}. Also, we observe that $w\in m$ for the same condition.

  \item[(b.1)] Consider $a$ such that $\{z\}\cup\{a\}=\emptyset$. By (a.1) we obtain that $a=w$. Therefore $a\in m$ and by Corollary \ref{coro 12 successor is non empty} we have that $a\in S(m)$.
\end{description}

  \item[(case 2)]Suppose that $z=n$.

\begin{description}
  \item[(a.2)] Since $z=n$ then it is clear, by the inductive hypothesis, that there exists a unique $m$ such that $\{n\}\cup\{m\}=\emptyset$.

  \item[(b.2)] Consider $a$ such that $\{n\}\cup\{a\}=\emptyset$. By (a.2) we obtain that $a=m$. Finally by Corollary \ref{coro 12 successor is non
  empty} we have that $a\in S(m)$.
\end{description}
\end{description}
Therefore we have that
$$\{S(n)\}\cup\{S(m)\}=\emptyset.$$

\textbf{Uniqueness}: This fact is obvious by Theorem \ref{theo 8 antielement is unique}.
\end{proof}

\begin{theorem}\label{theo 12 n star exists unique m antielem}
For all $n\in \mathbb{N}^{\star}$ there exists a unique $\sigma$-set $m$ such that $\{n\}\cup\{m\}=\emptyset$.
\end{theorem}

\begin{proof}
This proof is analogous to the previous one.
\end{proof}

We include the following Corollary in order to calculate the $\sigma$-anti-elements of the $\sigma$-elements of $\mathbb{N}$ and
$\mathbb{N}^{\star}$.

\begin{corollary}\label{coro 12}
The following statements holds:
\begin{description}
  \item[(a)] For all $n\in \mathbb{N}$, $S(n^{\st})=S(n)^{\st}$.
  \item[(b)] For all $n\in \mathbb{N}^{\star}$, $S(n^{\st})=S(n)^{\st}$.
\end{description}
\end{corollary}

\begin{proof}
We will only prove (a) because the proof of (b) is analogous.

Consider $n=1$. Then by Theorem \ref{theo 8 2 antielemt of 2 star} we obtain that $S(1^{\st})=S(1)^{\st}=2^{\st}$. Let $n\in\mathbb{N}$ such that $S(n^{\st})=S(n)^{\st}$. We will prove that $S(S(n)^{\st})=S(S(n))^{\st}$. Since $S(n^{\st})=S(n)^{\st}$ then
$$S(S(n)^{\st})=S(S(n))^{\st}\leftrightarrow S(S(n^{\st}))=S(S(n))^{\st}.$$
Therefore by Theorem \ref{theo 12 n exists a unique m star antiel} it is only necessary to prove that
$$\{S(S(n))\}\cup \{ S(S(n^{\st}))\}=\emptyset.$$
Since $\{S(n)\}\cup \{S(n^{\st})\}=\emptyset$ then by Lemmas \ref{lemma 12 x totdif y then s(x) totdif s(y)} and \ref{lemma 12 min(n) equal to min(n+1)}, we obtain that $S(S(n))$ and $S(S(n^{\st}))$ satisfy the conditions (a) and (b) of Axiom \ref{axiom of completeness b} (Completeness B). We only need to prove that conditions (c) and (d) of Axiom \ref{axiom of completeness b} are satisfied by them. Nevertheless, we only prove
condition (c) because the proof of condition (d) is analogous.

Let $z\in S(S(n))$ such that $z\notin \min(S(S(n)))=\min(S(n))=1$. By Corollary \ref{coro 12 successor is non empty} we have that $z\in S(n)$ or $z=S(n)$.

\begin{description}
  \item[(case 1)] Suppose that $z\in S(n)$.

\begin{description}
  \item[(a.1)] Since $z\in S(n)$ and $z\notin \min(S(S(n)))=\min(S(n))$, then we have that there exists a unique $w$ such that $\{z\}\cup\{w\}=\emptyset$ because $S(n)$ and $S(n^{\st})$ satisfy the condition (c) of Axiom \ref{axiom of completeness b}. Also, we observe that $w\in S(n^{\st})$, for the same reason.

  \item[(b.1)] Consider $a$ such that $\{z\}\cup\{a\}=\emptyset$. By (a.1) we obtain that $a=w$. Therefore $a\in S(n^{\st})$ and by Corollary \ref{coro 12 successor is non empty} we have that $a\in S(S(n^{\st}))$.
\end{description}

  \item[(case 2)]Suppose that $z=S(n)$.

\begin{description}
  \item[(a.2)] Since $z=S(n)$ then it is clear, by inductive hypothesis, that there exists a
unique $S(n^{\st})$ such that
$\{S(n)\}\cup\{S(n^{\st})\}=\emptyset$.

  \item[(b.2)] Consider $a$ such that $\{S(n)\}\cup\{a\}=\emptyset$. By (a.2) we obtain that $a=S(n^{\st})$. Finally by Corollary \ref{coro 12 successor is non
  empty} we have that $a\in S(S(n^{\st}))$.
\end{description}
\end{description}
Therefore we have that
$$\{S(S(n))\}\cup\{S(S(n^{\st}))\}=\emptyset.$$
\end{proof}

\begin{theorem}\label{theo 12 n nat then n star antinat}
The following statements holds:
\begin{description}
  \item[(a)] For all $n\in \mathbb{N}$, $n^{\st}\in \mathbb{N}^{\star}.$
  \item[(b)] For all $n\in \mathbb{N}^{\star}$, $n^{\st}\in \mathbb{N}.$
\end{description}
\end{theorem}

\begin{proof}
We will only prove (a) because the proof of (b) is analogous. If $n=1$ then it is clear that $(1)^{\st}=1^{\st}\in \mathbb{N}^{\star}$. Let $n\in \mathbb{N}$ such that $n^{\st}\in \mathbb{N}^{\star}$. Since $\mathbb{N}$ and $\mathbb{N}^{\star}$ are inductive $\sigma$-sets then
$S(n)\in \mathbb{N}$ and $S(n^{\st})\in \mathbb{N}^{\star}$. Finally, by Corollary \ref{coro 12} we have that $S(n^{\st})=S(n)^{\st}\in \mathbb{N}^{\star}$.
\end{proof}

We will use the following notation:
\begin{itemize}
  \item $n,m,i,j,k,$ etc. to denote natural numbers.
  \item $n^{\st},m^{\st},i^{\st},j^{\st},k^{\st},$ etc. to denote anti-natural numbers.
  \item $n_{\Theta},m_{\Theta},i_{\Theta},j_{\Theta},k_{\Theta},$ etc. to denote $\Theta$-natural numbers.

  \item If $n\in \mathbb{N}$ then $S(n)=n+1\in \mathbb{N}$,
  $$\mathbb{N}=\{1,2,3,4,\ldots\},$$
  $1$, $S(1)=1+1=2$, $S(2)=2+1=3$, $S(3)=3+1=4$, etc.

  \item If $n^{\st}\in \mathbb{N}^{\star}$ then
  $S(n^{\st})=n^{\st}+1^{\st}\in\mathbb{N}^{\star}$,
  $$\mathbb{N}^{\star}=\{1^{\st},2^{\st},3^{\st},4^{\st},\ldots\},$$
  $1^{\st}$, $S(1^{\st})=1^{\st}+1^{\st}=2^{\st}$, $S(2^{\st})=2^{\st}+1^{\st}=3^{\st}$, $S(3^{\st})=3^{\st}+1^{\st}=4^{\st}$, etc.

  \item If $n_{\Theta}\in \mathbb{N}_{\Theta}$ then
  $S(n_{\Theta})=n_{\Theta}+1_{\Theta}\in\mathbb{N}_{\Theta}$,
  $$\mathbb{N}_{\Theta}=\{1_{\Theta},2_{\Theta},3_{\Theta},4_{\Theta},\ldots\},$$
  $1_{\Theta}$, $S(1_{\Theta})=1_{\Theta}+1_{\Theta}=2_{\Theta}$, $S(2_{\Theta})=2_{\Theta}+1_{\Theta}=3_{\Theta}$, $S(3_{\Theta})=3_{\Theta}+1_{\Theta}=4_{\Theta}$, etc.
\end{itemize}

Therefore, by Corollary \ref{coro 12}, we obtain that
$$\{1\}\cup\{1^{\st}\}=\{2\}\cup\{2^{\st}\}=\{3\}\cup\{3^{\st}\}= \{4\}\cup\{4^{\st}\}=\ldots=\emptyset.$$

\end{subsection}

\begin{subsection}{The Axiom of Generated $\sigma$-set.}\label{axiom of generated set}
For all $\sigma$-sets $X$ and $Y$ there exists a $\sigma$-set, called the $\sigma$-set generated by $X$ and $Y$, whose $\sigma$-elements are exactly the fusion of the $\sigma$-subsets of $X$ with the $\sigma$-subsets of $Y$, that is
$$(\forall X,Y)(\exists Z)(\forall a)(a \in Z \leftrightarrow(\exists A\in 2^{X})(\exists B\in 2^{Y})(a=A\cup B)).$$

We can define the Generated Space by $X$ and $Y$.

\begin{definition}
Let $X$ and $Y$ be $\sigma$-sets. The \textbf{Generated Space by $X$ and $Y$} is given by
$$\langle 2^{X},2^{Y}\rangle=\{A\cup B: A\in 2^{X} \wedge B\in 2^{Y}\}.$$
\end{definition}

\begin{figure}
\centering
\includegraphics[width=1\textwidth]{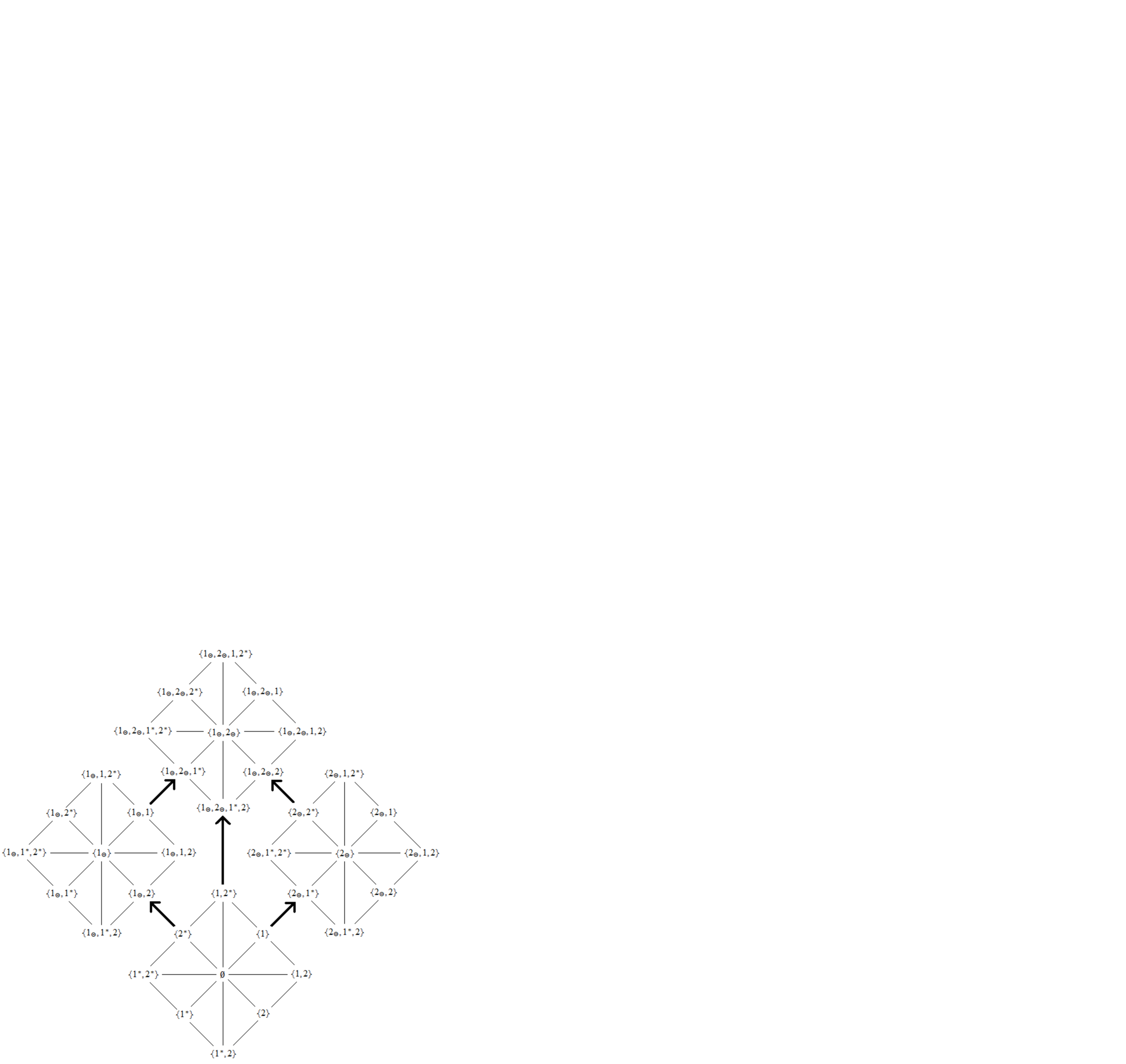}
\caption{Generated Space by $\{1_{\Theta},2_{\Theta},1,2\}$ and $\{1_{\Theta},2_{\Theta},1^{\st},2^{\st}\}$}
\label{generated space}
\end{figure}

We observe that in general $2^{A\cup B}\neq \langle 2^{A},2^{B}\rangle.$ Consider $X=\{1_{\Theta},2^{\st}\}$ and $Y=\{1_{\Theta},2\}$, then $X\cup Y=\{1_{\Theta}\}$. Therefore $2^{X\cup Y}=\{\emptyset, \{1_{\Theta}\}\}$ and $\langle 2^{X},2^{Y}\rangle=\{\emptyset, \{1_{\Theta}\}, \{2\}, \{2^{\st}\}, \{1_{\Theta},2\},\{1_{\Theta},2^{\st}\} \}.$

On the other hand, if we consider $A=\{1_{\Theta},2_{\Theta}\}$ and $B=\{1,2\}$ we obtain that the generated space by $A\cup B$
and $A\cup B^{\star}$ is the following:

\

$\langle 2^{A\cup B},2^{A\cup B^{\star}}\rangle=$ $\{\emptyset, \{1\},$ $\{1^{\st}\},$ $\{2\},$ $\{2^{\st}\},$ $\{1_{\Theta}\},$ $\{2_{\Theta}\},$ $\{1^{\st},2\},$ $\{1,2^{\st}\},$ $\{1,2\},$ $\{1^{\st},2^{\st}\},$ $\{1_{\Theta},1\},$ $\{1_{\Theta},1^{\st}\},$ $\{1_{\Theta},2\},$ $\{1_{\Theta},2^{\st}\},$ $\{2_{\Theta},1\},$ $\{2_{\Theta},1^{\st}\},$ $\{2_{\Theta},2\},$ $\{2_{\Theta},2^{\st}\},$ $\{1_{\Theta},2_{\Theta}\},$
$\{1_{\Theta},1^{\st},2\},$ $\{1_{\Theta},1,2^{\st}\},$ $\{1_{\Theta},1,2\},$ $\{1_{\Theta},1^{\st},2^{\st}\},$ $\{2_{\Theta},1^{\st},2\},$ $\{2_{\Theta},1,2^{\st}\},$ $\{2_{\Theta},1,2\},$ \ $\{2_{\Theta},1^{\st},2^{\st}\},$ \ $\{1_{\Theta},2_{\Theta},1\},$ $\{1_{\Theta},2_{\Theta},1^{\st}\},$ \ $\{1_{\Theta},2_{\Theta},2\},$ $\{1_{\Theta},2_{\Theta},2^{\st}\},$ $\{1_{\Theta},2_{\Theta},1^{\st},2\},$ $\{1_{\Theta},2_{\Theta},1,2^{\st}\},$ $\{1_{\Theta},2_{\Theta},1,2\},$ $\{1_{\Theta},2_{\Theta},1^{\st},2^{\st}\}\}.$

\

See Figure \ref{generated space}. 

\begin{definition}\label{def integer space}
Let $X$ and $Y$ be $\sigma$-sets. We say that $\langle 2^{X},2^{Y}\rangle$ is the Integer Space generated by $X$ if $Y$ is the $\sigma$-antiset of $X$ (i.e. $X^{\star}=Y$). In this case the Integer Space is denoted by
$$3^{X}=\langle 2^{X},2^{X^{\star}}\rangle.$$
\end{definition}

If $A=\{1,2,3\}$ and $A^{\star}=\{1^{\st},2^{\st},3^{\st}\}$, then

\

$\langle 2^{A},2^{A^{\star}}\rangle=3^{A}$ $=\{\emptyset,$ $\{1\},$ $\{2\},$ $\{3\},$ $\{1^{\st}\},$ $\{2^{\st}\},$ $\{3^{\st}\},$ $\{1,2\},$ $\{1,3\},$ $\{2,3\},$ $\{1^{\st},2\},$ $\{1^{\st},3\},$ $\{2^{\st},3\},$ $\{1^{\st},2^{\st}\},$ $\{1^{\st},3^{\st}\},$ $\{2^{\st},3^{\st}\},$ $\{1,2^{\st}\},$ $\{1,3^{\st}\},$ $\{2,3^{\st}\},$ $\{1,2,3\},$ \ $\{1^{\st},2,3\},$ \ $\{1,2^{\st},3\},$ \ $\{1,2,3^{\st}\},$ $\{1^{\st},2^{\st},3\},$ $\{1^{\st},2,3^{\st}\},$ $\{1,2^{\st},3^{\st}\},$ $\{1^{\st},2^{\st},3^{\st}\} \}.$

\

Now, in Figure \ref{integer space}, we present the patterns of containments of the $3^{X}$. We observe that the $\sigma$-elements $\{1^{\st},2,3^{\st}\}$ and $\{1,2^{\st},3\}$ can be represented in three dimensions as one of
the vertexes of the pyramids
$$\triangle_{1}:=\{\{1,3\},\{2^{\st},3\},\{1,2^{\st}\}, \{1,2^{\st},3\}\}$$
and
$$\triangle_{2}:= \{\{1^{\st},3^{\st}\},\{2,3^{\st}\},\{1^{\st},2\},\{1^{\st},2,3^{\st}\}\}.$$

\begin{figure}
\centering
\includegraphics[width=0.75\textwidth]{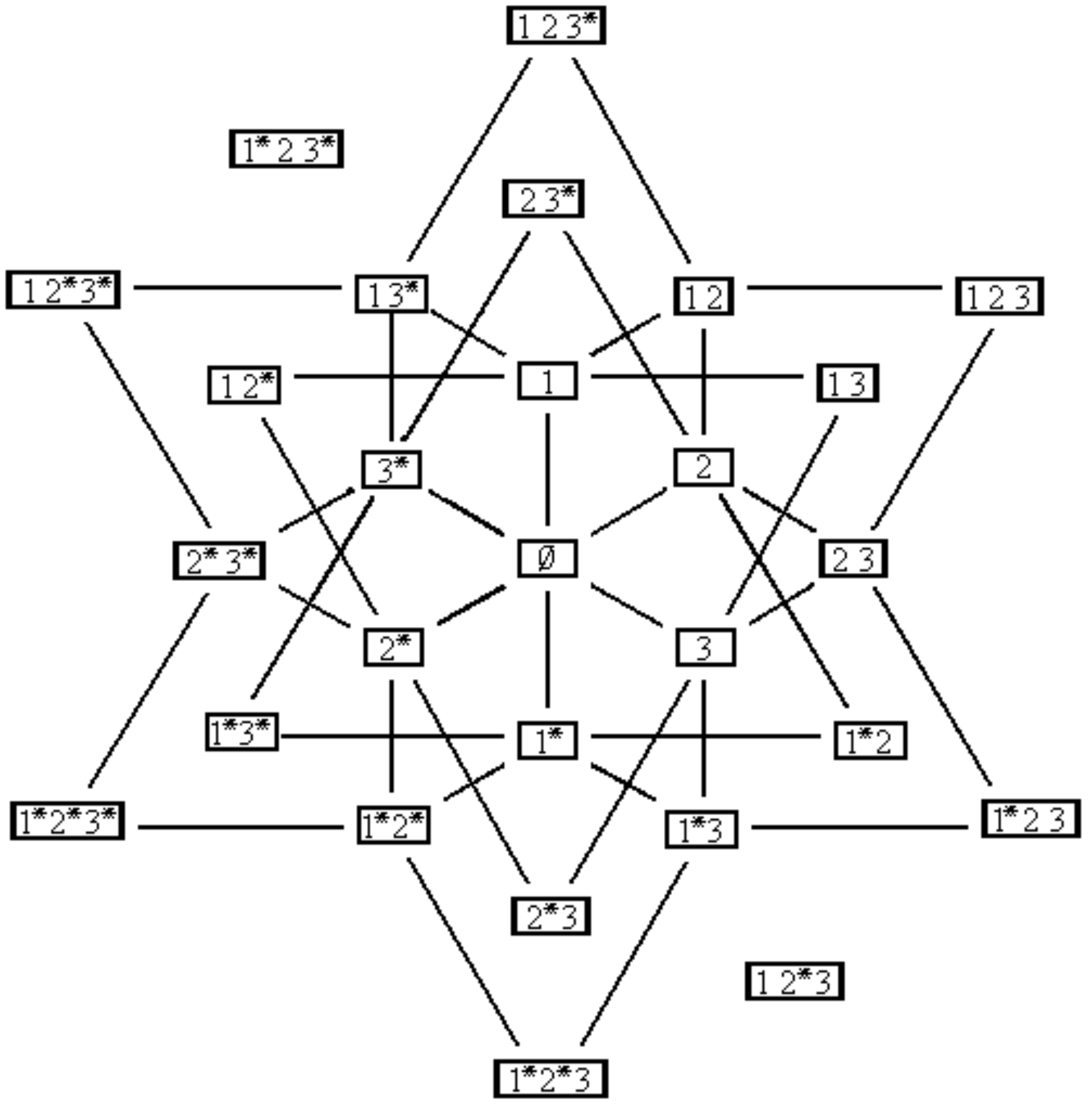}
\caption{Integer Space $3^{\{1,2,3\}}$}
\label{integer space}
\end{figure}

\end{subsection}
\end{section}

\begin{section}{Final Comments.}

The study of axiomatic system of $\sigma$-Set Theory leads, among other things, to the construction of $\mathbb{N}$, $\mathbb{N}^{\star}$ and $\mathbb{N}_{\Theta}$. It is important to notice some of the basic properties of these $\sigma$-sets.

By definition we have that:
\begin{itemize}
	\item $\mathbb{N}=\{1,2,3,4,5,\ldots\}$;
	\item $\mathbb{N}^{\star}=\{1^{\st},2^{\st},3^{\st},4^{\st},5^{\st},\ldots\}$;
	\item $\mathbb{N}_{\Theta}=\{1_{\Theta},2_{\Theta},3_{\Theta},4_{\Theta},5_{\Theta},\ldots\}$.
\end{itemize}
If we consider $n,m\in\mathbb{N}$ such that $n\neq m$, then we have that
\begin{itemize}
	\item $\{n\}\cup\{n^{\st}\}=\{n^{\st}\}\cup\{n\}=\emptyset$;
	\item $\{n\}\cup\{m^{\st}\}=\{m^{\st}\}\cup\{n\}=\{n,m^{\st}\}$;
	\item $\{n\}\cup\{n_{\Theta}\}=\{n_{\Theta}\}\cup\{n\}=\{n,n_{\Theta}\}$;
	\item $\{n^{\st}\}\cup\{n_{\Theta}\}=\{n_{\Theta}\}\cup\{n^{\st}\}=\{n^{\st},n_{\Theta}\}$;
	\item $\{n\}\cup\{m_{\Theta}\}=\{m_{\Theta}\}\cup\{n\}=\{n,m_{\Theta}\}$;
	\item $\{n^{\st}\}\cup\{m_{\Theta}\}=\{m_{\Theta}\}\cup\{n^{\st}\}=\{n^{\st},m_{\Theta}\}$.
\end{itemize}
Hence, a direct consequence of these properties is that
\begin{itemize}
	\item $\mathbb{N}\cup \mathbb{N}^{\star}=\emptyset$;
	\item $\mathbb{N}\cup \mathbb{N}_{\Theta}=\{1,1_{\Theta},2,2_{\Theta},3,3_{\Theta},4,4_{\Theta},5,5_{\Theta},\ldots\}$;
	\item $\mathbb{N}^{\star}\cup \mathbb{N}_{\Theta}=\{1^{\st},1_{\Theta},2^{\st},2_{\Theta},3^{\st},3_{\Theta},4^{\st},4_{\Theta},5^{\st},5_{\Theta},\ldots\}$.
\end{itemize}

In this paper, we give only the mathematical foundations of $\sigma$-Set Theory, thus, in the case that the axiomatic system will be consistent, in future work we will see the scope and limitations of the theory. However, we think that all definitions and theorems which are studied in a standard Set Theory, can be built on the $\sigma$-Set Theory. On the other hand, it is clear that the reciprocal is false because in a standard set theory, a set does not have the inverse element for the union operation.

Also,  we think that the study of the properties of the space generated by two $\sigma$-sets, as the study of power set in standard set theory, will lead to interesting results about the cardinality of a $\sigma$-set. Thus, we will be able to establish new relationships and properties of infinite cardinals. We present the following conjecture: If $X$ is a $\sigma$-set and $|X|$ is the cardinal of $X$ then we have that

\begin{conjecture}
For all $X\in 2^{\mathbb{N}_{\Theta}}$ and $Y\in 2^{\mathbb{N}}$, then
	$$|\langle 2^{X\cup Y},2^{X\cup Y^{\star}}\rangle|=2^{|X|}3^{|Y|}.$$
\end{conjecture}

\

Hence, if $X=\emptyset$, then $|\langle 2^{X\cup Y},2^{X\cup Y^{\star}}\rangle|=|3^{Y}|=3^{|Y|}$. On the other hand, if $Y=\emptyset$, then
$|\langle 2^{X\cup Y},2^{X\cup Y^{\star}}\rangle|=|2^{X}|=2^{|X|}$, because $\emptyset=\emptyset^{\star}$.

\begin{example}
We consider $X=\{1_{\Theta}\}$ and $Y=\{1,2\}$, then we obtain that

\begin{center}
\begin{tabular}{rl}
$\langle 2^{X\cup Y},2^{X\cup Y^{\star}}\rangle$ & $=\{\emptyset, \{1_{\Theta}\}, \{1\}, \{2\}, \{1^{\st}\}, \{2^{\st}\}, \{1_{\Theta},1\},\{1_{\Theta},2\}$\\
& $ \ \ \ \ \{1_{\Theta},1^{\st}\},\{1_{\Theta},2^{\st}\},\{1^{\st},2\},\{1,2^{\st}\}, \{1,2\},\{1^{\st},2^{\st}\},$\\
& $ \ \ \ \ \{1_{\Theta},1^{\st},2\},\{1_{\Theta},1,2^{\st}\}, \{1_{\Theta},1,2\},\{1_{\Theta},1^{\st},2^{\st}\}\}. $\\
\end{tabular}
\end{center}

Thus, we have that $|X|=1$, $|Y|=2$ and $|\langle 2^{X\cup Y},2^{X\cup Y^{\star}}\rangle|=2^{1}\cdot 3^{2}=18$.
\end{example}

\begin{example}\label{example 4.2}
We consider $X=\emptyset$ and $Y=\{1,2\}$, then we obtain that

$$3^{Y}=\{\emptyset, \{1\}, \{2\}, \{1^{\st}\}, \{2^{\st}\}, \{1^{\st},2\},\{1,2^{\st}\}, \{1,2\},\{1^{\st},2^{\st}\}\}.$$

Hence, we have that $|X|=0$, $|Y|=2$ and $|\langle 2^{X\cup Y},2^{X\cup Y^{\star}}\rangle|=2^{0}\cdot 3^{2}=9$.
\end{example}

With respect to the algebraic structure of the Integer Space $3^{X}$ and generated space $\langle 2^{X},2^{Y}\rangle$, we think that these structures are related with structures called NAFIL (non-associative finite invertible loops). We present our conjecture. If we define
$$A\oplus B\leftrightarrow A\cup B,$$
then we obtain the following.

\begin{conjecture}\label{conjecture 4.2}
For all $X\in 2^{IN}$ we have that $(3^{X},\oplus)$ satisfies the following conditions:
\begin{enumerate}
	\item $(\forall A,B\in 3^{X})(A\oplus B\in 3^{X})$.
	\item $(\exists !\xi\in 3^{X})(\forall A\in 3^{X})(\xi\oplus A=A\oplus\xi=A)$.
	\item $(\forall A\in 3^{X})(\exists !A^{\star}\in 3^{X})(A\oplus A^{\star}=A^{\star}\oplus A=\xi)$.
	\item $(\forall A,B\in 3^{X})(A\oplus B=B\oplus A)$.
\end{enumerate}
\end{conjecture}

It is clear, by Example \ref{example 4.2} that $3^{Y}$ satisfies the conditions (1),(2),(3) and (4) from Conjecture \ref{conjecture 4.2}.

Following with the possible applications of the $\sigma$-Set Theory and taking into account the developments of Lattices, we define the following relation: If we consider $A,B\in 3^{X}$ then we define
$$A\leq B\leftrightarrow B\oplus A^{\star}\in 2^{X}.$$
Thus, we obtain the following conjecture:

\begin{conjecture}\label{conjecture 4.3}
For all $X\in 2^{IN}$ we have that $(3^{X},\leq)$ satisfies the following conditions:
\begin{enumerate}
	\item $(\forall A\in 3^{X})(A\leq A)$.
	\item $(\forall A,B\in 3^{X})(A\leq B\wedge B\leq A \rightarrow A=B)$.
	\item $(\forall A,B\in 3^{X})(A\leq B\wedge B\leq C \rightarrow A\leq C)$.
\end{enumerate}
\end{conjecture}

Therefore, for all $X\in 2^{IN}$ we have that $(3^{X},\leq)$ is a\textbf{ partially ordered $\sigma$-set} and $2^{X}$ represent the \textbf{positive cone} of $3^{X}$. That is
$$2^{X}=\{x\in 3^{X}: x\geq\emptyset\}.$$

\begin{example}
We consider $X=\{1,2,3\}$ and the Integer Space $3^{X}$ (see example \ref{example 3.77}). If we consider the relation
$$A\leq B\leftrightarrow B\oplus A^{\star}\in 2^{X},$$
where $2^{X}=\{\emptyset,\{1\},\{2\},\{3\},\{1,2\},\{1,3\},\{2,3\},\{1,2,3\}\}$, then we have that
$$2^{X}=\{x\in 3^{X}: x\geq\emptyset\}.$$
Hence, $(3^{X},\leq)$ satisfies the conditions (1),(2) and (3) from Conjecture \ref{conjecture 4.3}. Therefore, we have that
$$\{1^{\st},2^{\st},3^{\st}\}\leq\{1^{\st},2^{\st}\}\leq\{1^{\st}\}\leq\emptyset\leq\{1\}\leq\{1,2\}\leq\{1,2,3\},$$
$$\{1^{\st},2^{\st},3^{\st}\}\leq\{1^{\st},3^{\st}\}\leq\{1^{\st}\}\leq\emptyset\leq\{1\}\leq\{1,3\}\leq\{1,2,3\},$$
$$\{1^{\st},2^{\st},3^{\st}\}\leq\{1^{\st},2^{\st}\}\leq\{2^{\st}\}\leq\emptyset\leq\{2\}\leq\{1,2\}\leq\{1,2,3\},$$
$$\{1^{\st},2^{\st},3^{\st}\}\leq\{2^{\st},3^{\st}\}\leq\{2^{\st}\}\leq\emptyset\leq\{2\}\leq\{2,3\}\leq\{1,2,3\},$$
$$\{1^{\st},2^{\st},3^{\st}\}\leq\{1^{\st},3^{\st}\}\leq\{3^{\st}\}\leq\emptyset\leq\{3\}\leq\{1,3\}\leq\{1,2,3\},$$
$$\{1^{\st},2^{\st},3^{\st}\}\leq\{2^{\st},3^{\st}\}\leq\{3^{\st}\}\leq\emptyset\leq\{3\}\leq\{2,3\}\leq\{1,2,3\}.$$
Also, we observe, for example, that $A=\{2^{\st},3^{\st}\}$ and $B=\{1^{\st},3^{\st}\}$ are not comparable, because $A\oplus B^{\star}=\{1,2^{\st}\}$ and $B\oplus A^{\star}=\{1^{\st},2\}$.
\end{example}

Therefore, if we follow the definition of basic operations + and - for real numbers, we can define $B\ominus A:= B\oplus A^{\star}$. So we obtain that
$$A\leq B \ \leftrightarrow \ \emptyset\leq B\oplus A^{\star} \ \leftrightarrow \ \emptyset\leq B\ominus A.$$

Finally, we present the summary of the axiomatic system of $\sigma$-Set Theory.

\begin{subsection}{The Axiom of Empty $\sigma$-set.}\label{3 axiom empty set}
There exists a $\sigma$-set which has no $\sigma$-elements, that is
$$(\exists X)(\forall x)(x\notin X).$$
\end{subsection}

\begin{subsection}{The Axiom of Extensionality.}\label{3 axiom of extensionality}
For all $\sigma$-classes $\hat{X}$ and $\hat{Y}$, if $\hat{X}$ and $\hat{Y}$ have the same
$\sigma$-elements, then $\hat{X}$ and $\hat{Y}$ are equal, that is
$$(\forall \hat{X},\hat{Y})[(\forall z)(z\in \hat{X} \leftrightarrow z\in \hat{Y})\rightarrow \hat{X}=\hat{Y}].$$
\end{subsection}

\begin{subsection}{The Axiom of Creation of $\sigma$-Class.}\label{3 axiom of creation of class}
We consider an atomic formula $\Phi(x)$ (where $\hat{Y}$ is not free). Then there exists the class $\hat{Y}$ of all $\sigma$-sets that satisfy $\Phi(x)$, that is
$$(\exists\hat{Y})(x\in \hat{Y} \leftrightarrow \Phi(x)),$$
with $\Phi(x)$ being an atomic formula where $\hat{Y}$ is not free.
\end{subsection}

\begin{subsection}{The Axiom of Scheme of Replacement.}\label{3 axiom of replacement}
The image of a $\sigma$-set under a normal functional formula $\Phi$ is a
$\sigma$-set.
\end{subsection}

\begin{subsection}{The Axiom of Pairs.}\label{3 axiom of pairs}
For all $X$ and $Y$ $\sigma$-sets there exists a $\sigma$-set $Z$, called fusion
of pairs of $X$ and $Y$, that satisfies one and only one of the
following conditions:
\begin{description}
  \item[(a)] $Z$ contains exactly $X$ and $Y$,
  \item[(b)] $Z$ is equal to the empty $\sigma$-set,
\end{description}
that is
$$(\forall X,Y)(\exists Z)(\forall a)[(a\in Z \leftrightarrow a=X \vee a=Y ) \underline{\vee}(a\notin Z)].$$
\end{subsection}

\begin{subsection}{The Axiom of Weak Regularity.}\label{3 axiom of w-regularity}
For all $\sigma$-set $X$, for all $\lch x,\ldots,w\rch\in CH(X)$ we have that $X\not\inch\lch x,\ldots,w\rch$, that
is
$$(\forall X)(\forall\lch x,\ldots,w\rch\in CH(X))(X\not\inch\lch x,\ldots,w\rch).$$
\end{subsection}

\begin{subsection}{The Axiom of non $\epsilon$-Bounded $\sigma$-Set.}\label{3 axiom non bounded set}
There exists a non $\epsilon$-bounded $\sigma$-set, that is
$$(\exists X)(\exists y)[(y\in X)\wedge(\min(X)=\emptyset\vee\max(X)=\emptyset)].$$
\end{subsection}

\begin{subsection}{The Axiom of Weak Choice.}\label{3 axiom of weak choice}
If $\hat{X}$ is a $\sigma$-class of $\sigma$-sets, then we can choose a singleton $Y$ whose unique $\sigma$-element come from $\hat{X}$, that is
$$(\forall\hat{X})(\forall x)(x\in\hat{X}\rightarrow (\exists Y)(Y=\{x\})).$$
\end{subsection}

\begin{subsection}{The Axiom of $\epsilon$-Linear $\sigma$-set.}\label{3 axiom e-linear set}
There exists a nonempty $\sigma$-set $X$ such that $X$ has the linear $\epsilon$-root property, that is
$$(\exists X)(\exists y)(y\in X\wedge X\in LR).$$
\end{subsection}

\begin{subsection}{The Axiom of totally different $\sigma$-sets.}\label{3 axiom one set}
For all $\epsilon$-linear singleton, there exists a $\epsilon$-linear singleton $Y$ such that $X$ is totally different from $Y$, that is
$$(\forall X\in SG\cap LR)(\exists Y\in SG\cap LR)(X\totdif Y).$$
\end{subsection}

\begin{subsection}{The Axiom of Completeness (A).}\label{3 axiom of completeness a}
If $X$ and $Y$ are $\sigma$-sets, then
$$\{X\}\cup\{Y\}=\{X,Y\},$$
if and only if $X$ and $Y$ satisfy one of the following conditions:
\begin{description}
  \item[(a)] $\min(X,Y)\neq |1\vee 1^{\st}|\wedge \min(X,Y)\neq |1^{\st}\vee 1|.$
  \item[(b)] $\neg(X\totdif Y).$
  \item[(c)] $(\exists z\in X)[z\notin \min(X)\wedge \neg \Psi(z,w,a,Y)].$
  \item[(d)] $(\exists z\in Y)[z\notin \min(Y)\wedge \neg \Psi(z,w,a,X)].$
\end{description}
\end{subsection}

\begin{subsection}{The Axiom of Completeness (B).}\label{3 axiom of completeness b}
If $X$ and $Y$ are $\sigma$-sets, then
$$\{X\}\cup\{Y\}=\emptyset,$$
if and only if $X$ and $Y$ satisfy the
following conditions:

\begin{description}
  \item[(a)] $\min(X,Y)=|1\wedge 1^{\st}|\vee \min(X,Y)=|1^{\st}\wedge 1|$;
  \item[(b)] $X\totdif Y;$
  \item[(c)] $(\forall z)(z\in X\wedge z\notin \min(X))\rightarrow \Psi(z,w,a,Y))$;
  \item[(d)] $(\forall z)(z\in Y\wedge z\notin \min(Y))\rightarrow \Psi(z,w,a,X))$.
\end{description}
\end{subsection}

\begin{subsection}{The Axiom of Exclusion.}\label{3 axiom of exclusion}
For all $\sigma$-sets $X,Y,Z$, if $Y$ and $Z$ are $\sigma$-elements of $X$ then the fusion of pairs of $Y$ and $Z$ contains exactly $Y$ and $Z$, that is
$$(\forall X,Y,Z)(Y,Z\in X\rightarrow \{Y\}\cup\{Z\}=\{Y,Z\}).$$
\end{subsection}

\begin{subsection}{The Axiom of Power $\sigma$-set.}\label{3 axiom of power set}
For all $\sigma$-set $X$  there exists a $\sigma$-set $Y$, called power of $X$, whose $\sigma$-elements are exactly the $\sigma$-subsets of $X$, that is
$$(\forall X)(\exists Y)(\forall z)(z\in Y\leftrightarrow z\subseteq X).$$
\end{subsection}

\begin{subsection}{The Axiom of Fusion.}\label{3 axiom of fusion}
For all $\sigma$-sets $X$ and $Y$, there exists a $\sigma$-set $Z$, called fusion of all $\sigma$-elements of $X$ and $Y$, such that $Z$ contains $\sigma$-elements of the $\sigma$-elements of $X$ or $Y$, that is
$$(\forall X,Y)(\exists Z)(\forall b)(b \in Z \rightarrow (\exists z)[(z \in X\vee z\in Y)\wedge(b\in z)]).$$
\end{subsection}

\begin{subsection}{The Axiom of Generated $\sigma$-set.}\label{3 axiom of generated set}
For all $\sigma$-sets $X$ and $Y$ there exists a $\sigma$-set, called the $\sigma$-set generated by $X$ and $Y$, whose $\sigma$-elements are exactly the fusion of the $\sigma$-subsets of $X$ with the $\sigma$-subsets of $Y$, that is
$$(\forall X,Y)(\exists Z)(\forall a)(a \in Z \leftrightarrow(\exists A\in 2^{X})(\exists B\in 2^{Y})(a=A\cup B)).$$
\end{subsection}
\end{section}

%%%%%%%%%%%%%%%%
% bibliography
%%%%%%%%%%%%%%%

% Set bibliography items using the "thebibliography" environment  and following
% the style used by the AMS journals.
%
% If the bibliography is generated by a bibtex database, use "amsplain" or
% "amsalpha" as bibliography style

\end{document}